\documentclass[a4paper,11pt,oneside]{article}

\usepackage[usenames,dvipsnames,svgnames,table]{xcolor}
\usepackage[english]{babel}
\usepackage[hidelinks]{hyperref}
\usepackage{amsmath,amssymb,amsthm}
\usepackage{mathtools}
\usepackage[top=1in, bottom=1in, left=1in, right=1in]{geometry}

\usepackage{thmtools}
\usepackage{thm-restate}

\makeatletter
\newcommand{\mypm}{\mathbin{\mathpalette\@mypm\relax}}
\newcommand{\@mypm}[2]{\ooalign{%
  \raisebox{.1\height}{$#1+$}\cr
  \smash{\raisebox{-.6\height}{$#1-$}}\cr}}
\makeatother

\newcommand{\N}{\mathbb{N}}
\let\OLDthebibliography\thebibliography
\renewcommand\thebibliography[1]{
  \OLDthebibliography{#1}
  \setlength{\parskip}{0pt}
  \setlength{\itemsep}{0pt plus 0.3ex}
}

\DeclareMathOperator{\rank}{rank}

\newcommand{\ignore}[1]{}

\newcommand{\cpsdrank}{\text{\rm cpsd-rank}}
\newcommand{\xib}[2]{{\xi_{#2}^{\mathrm{#1}}}}
\newcommand{\las}[2]{{\mathrm{las}_{#2}^{\mathrm{#1}}}}

\newcommand{\Cq}{C_{q}}
\newcommand{\Cqc}{C_{qc}}
\newcommand{\Cloc}{C_{loc}}

\renewcommand{\S}{\mathcal{S}}

\newcommand{\R}{\mathbb{R}}

\newcommand{\C}{\mathbb{C}}

\newcommand{\T}{{\sf T}}
\DeclareMathOperator{\Tr}{Tr}

\renewcommand{\sup}{\mathrm{sup}}
\renewcommand{\min}{\mathrm{min}}

\renewcommand{\max}{\mathrm{max}}

\renewcommand{\epsilon}{\varepsilon}

\newtheorem{defin}{Definition}[section]
\newtheorem{definition}[defin]{Definition}
\newtheorem{proposition}[defin]{Proposition}
\newtheorem{theorem}[defin]{Theorem}
\newtheorem{remark}[defin]{Remark}

\newtheorem{lemma}[defin]{Lemma}

\newtheorem*{claim*}{Claim}

\newtheorem*{conjecture*}{Conjecture}
\theoremstyle{definition}

\DeclareMathOperator{\cpsdr}{cpsd-rank}

\newcommand{\cx}{[\mathbf{x}]}
\newcommand{\ncx}{\langle {\bf x}\rangle}
\newcommand{\bx}{\mathbf x}
\newcommand{\by}{\mathbf y}
\newcommand{\bX}{\mathbf X}

\newcommand{\MM}{{\mathcal M}}
\newcommand{\MI}{{\mathcal I}}

\newcommand{\MB}{{\mathcal B}}

\newcommand{\MA}{{\mathcal A}}

\date{}
\begin{document}

\title{Bounds on entanglement dimensions and quantum graph parameters via noncommutative polynomial optimization}
\author{
Sander Gribling\thanks{CWI and QuSoft, Amsterdam, the Netherlands. Supported by the Netherlands Organization for Scientific Research, grant number 617.001.351. {\tt gribling@cwi.nl}}
\and
David de Laat\thanks{CWI and  QuSoft, Amsterdam, the Netherlands. Supported by the Netherlands Organization for Scientific Research, grant number 617.001.351, and by the ERC Consolidator Grant QPROGRESS 615307. {\tt mail@daviddelaat.nl}} 
\and
Monique Laurent\thanks{CWI and  QuSoft, Amsterdam, and Tilburg University, Tilburg,  the Netherlands. {\tt laurent@cwi.nl}}}


\maketitle

\begin{abstract}
In this paper we study optimization problems related to bipartite quantum correlations
using techniques from tracial noncommutative polynomial optimization. First we consider the problem
of finding the minimal entanglement dimension of such correlations. We construct
a hierarchy of semidefinite programming lower bounds and show convergence to a
new parameter: the minimal average entanglement dimension, which measures the
amount of entanglement needed to reproduce a quantum correlation when access to
shared randomness is free. Then we study optimization problems over synchronous
quantum correlations arising from quantum graph parameters. We introduce semidefinite programming hierarchies and unify existing bounds on quantum chromatic
and quantum stability numbers by placing them in the framework of tracial optimization.
\end{abstract}

\section{Introduction}

\subsection{Bipartite quantum correlations} \label{sec:quantumcorrelations}

One of the distinguishing features of quantum mechanics is quantum entanglement, which allows for nonclassical correlations between spatially separated parties. 
In this paper we consider the problems of quantifying the advantage entanglement can bring (first investigated through Bell inequalities in the seminal work~\cite{B64}) and quantifying the minimal amount of entanglement necessary for generating a given correlation (initiated in~\cite{Brunner08} and continued, e.g.,  in~\cite{PV08,WCD08,SVW16}). 

Quantum entanglement has been widely studied in the bipartite correlation setting (for a survey, see, e.g.,~\cite{PV16}). Here we have two parties, Alice and Bob, where Alice receives a question $s$ taken from a finite set $S$ and Bob receives a question $t$ taken from a finite set $T$. The parties do not know each other's questions, and after receiving the questions they do not communicate. Then, according to some predetermined protocol, Alice returns an answer $a$ from a finite set $A$ and Bob returns an answer $b$ from a finite set $B$. The probability that the parties answer $(a,b)$ to questions $(s,t)$ is given by a \emph{bipartite correlation} $P(a,b|s,t)$, which satisfies $P(a,b|s,t) \geq 0$ for all $(a,b,s,t)\in \Gamma$ and $\sum_{a,b} P(a,b|s,t) = 1$ for all $(s,t)\in S\times T$.
Throughout we set $\Gamma =A\times B \times S \times T$.
Which bipartite correlations $P=(P(a,b|s,t))\in \smash{\R^\Gamma}$  are possible depends on the additional resources  available to the two parties Alice and Bob. 

If the parties do not have access to additional resources, then the correlation is \emph{deterministic}, which means it is of the form $P(a,b|s,t) = P_A(a|s) \, P_B(b|t)$, with $P_A(a|s)$ and $\smash{P_B(b|t)}$ taking values in $\{0,1\}$ and $\sum_a P_A(a|s) = \sum_b P_B(b|t) = 1$ for all $s,t$. If the parties have access to local randomness, then $P_A$ and $P_B$ take values in $[0,1]$. If the parties have access to shared randomness, 
then the resulting correlation is a convex combination of deterministic correlations and is said to be a \emph{classical correlation}. The classical correlations form a polytope, denoted $\Cloc(\Gamma)$, whose 
valid inequalities are known as {\em Bell inequalities}~\cite{B64}.

We are interested in the quantum setting, where the parties have access to a shared quantum state on which they can perform measurements. 
The quantum setting can be modeled in different ways, leading to the so-called tensor and commuting models; see the discussion, e.g., in~\cite{Tsirelson06,NPA08,DLTW08}.

In the {\em tensor model},  Alice and Bob each have access to ``one half" of a finite dimensional \emph{quantum state}, which is modeled by a unit vector $\psi \in \C^d \otimes \C^d$.
Alice and Bob determine their answers by performing a measurement on their part of the state. Such a measurement is modeled by a {\em positive operator valued measure} (POVM), which consists of a set of $d \times d$ Hermitian positive semidefinite matrices labeled by the possible answers and summing to the identity matrix. If Alice uses the POVM $\{E_s^a\}_{a \in A}$ when she gets question $s \in S$ and Bob uses the POVM $\{F_t^b\}_{b \in B}$ when he gets question $t \in T$,  then the probability of obtaining the answers $(a,b)$ is given by
\begin{equation}\label{eqPtensor}
P(a,b|s,t) =  \mathrm{Tr}( (E_s^a \otimes F_t^b)  \psi \psi^*) = \psi^* (E_s^a \otimes F_t^b) \psi.
\end{equation}
If the state $\psi$ cannot be written as a single tensor product $\psi_A \otimes \psi_B$, then $\psi$ is \emph{entangled}, which means it can be used to produce a nonclassical correlation $P$.

A correlation of the above form~\eqref{eqPtensor} is a \emph{quantum correlation}, realizable in the tensor model in \emph{local dimension} $d$ (or in \emph{dimension} $d^2$). Let $\smash{\Cq^d(\Gamma)}$ be the set of such correlations and define
\[
\Cq(\Gamma)=\bigcup_{d\in \N} \Cq^{d}(\Gamma).
\]
Denote the smallest dimension needed to realize 
$P \in \Cq(\Gamma)$ in the tensor model by
\begin{equation}\label{eqDq}
D_q(P) = \min \big\{ d^2  : d \in \N, \, P \in \Cq^d(\Gamma)\big\}.
\end{equation}

The set $\Cq^1(\Gamma)$ contains the deterministic correlations. Hence, by Carath\'eodory's theorem,  $\Cloc(\Gamma)\subseteq \Cq^c(\Gamma)$ holds for $c =  |\Gamma|+1-|S||T|$; that is, quantum entanglement can be used as an alternative to shared randomness. If $A$, $B$, $S$, and $T$ all contain at least two elements, then Bell~\cite{B64}
shows the inclusion $\Cloc(\Gamma) \subseteq \Cq(\Gamma)$ is strict; that is, quantum entanglement can be used to obtain nonclassical correlations. 

The second commonly used model to define quantum correlations is the {\em commuting model} (or {\em relativistic field theory model}). 
Here a correlation $P\in \R^\Gamma$ is called a  \emph{commuting quantum correlation} if it is  of the form 
\begin{equation}\label{eqPcommute} 
P(a,b|s,t) =  \mathrm{Tr}(X_s^a Y_t^b \psi \psi^*) = \psi^* (X_s^a Y_t^b) \psi,
\end{equation}
where $\{X_s^a\}_a$ and $\{Y_t^b\}_b$ are POVMs consisting of bounded operators on a separable Hilbert space $H$, satisfying $[X_s^a, Y_t^b] = X_s^a Y_t^b - Y_t^bX_s^a = 0$ for all $(a,b,s,t)\in\Gamma$, and where $\psi$ is a unit vector in $H$. Such a correlation is said to be realizable  in dimension $d = \mathrm{dim}(H)$ in the commuting model. Denote the set of such correlations by $\Cqc^d(\Gamma)$ and set $\Cqc(\Gamma) =\Cqc^\infty(\Gamma)$. The smallest dimension needed to realize a quantum correlation $P \in \Cqc(\Gamma)$ is given by
\begin{equation}\label{eqDqc}
D_{qc}(P) = \min \big\{ d \in \N \cup \{\infty\} : P \in \Cqc^d(\Gamma)\big\}.
\end{equation}
We have $\Cq^d(\Gamma) \subseteq \smash{\Cqc^{d^2}}(\Gamma)$, which follows by setting $X_s^a = E_s^a \otimes I$ and $Y_t^b = I \otimes F_t^b$. This shows 
$D_{qc}(P) \leq D_q(P)$ for all $P \in \Cq(\Gamma)$.

The minimum Hilbert space dimension in which a given quantum correlation $P$ can be realized quantifies the minimal amount of entanglement needed to represent  $P$. Computing $D_q(P)$ is NP-hard~\cite{Stark15}, so a natural  question is to find good lower bounds for the parameters $D_q(P)$ and $D_{qc}(P)$. A main contribution of this paper is proposing a  hierarchy of semidefinite programming lower bounds for these parameters.

As said above we have $\Cq^d(\Gamma) \subseteq \smash{\Cqc^{d^2}}(\Gamma)$.
Conversely, each finite dimensional commuting quantum correlation can be realized in the tensor model, although not necessarily in the same dimension~\cite{Tsirelson06} (see, e.g.,~\cite{DLTW08} for a proof).  This shows 
\[
\Cq(\Gamma) = \smash{\bigcup_{d\in \N}} \Cqc^{d}(\Gamma)\subseteq \Cqc(\Gamma).
\]
Using a direct sum construction one can show $\cup_{d\in \N} \Cqc^{d}(\Gamma)$ and $\Cqc(\Gamma)$ are convex. Whether the two sets $\Cq(\Gamma)$ and $\Cqc(\Gamma)$ coincide is known as Tsirelson's problem. In a recent breakthrough Slofstra~\cite{Slofstra17} showed that 
$\Cq(\Gamma)$ is not closed for $|A| \geq 8$, $|B| \geq 2$, $|S| \geq 184$, $|T| \geq 235$.
More recently it was shown in~\cite{DPP17} that the same holds for $|A| \geq 2$, $|B| \geq 2$, $|S| \geq 5$, $|T| \geq 5$.
Hence, for such $\Gamma$ there is a sequence $\{P_i\} \subseteq  \Cq(\Gamma)$ with $D_q(P_i) \to \infty$. Moreover, since $\Cqc(\Gamma)$ is closed~\cite[Prop.~3.4]{Fri12}, 
the inclusion $\Cq(\Gamma) \subseteq \Cqc(\Gamma)$ is strict, thus settling Tsirelson's problem.
Whether the closure of $\Cq(\Gamma)$ equals $\Cqc(\Gamma)$ for all $\Gamma$ is equivalent to Connes' embedding conjecture in operator theory~\cite{JNPPGSW,Oz12}.

Further variations on the above definitions are possible. For instance, we can consider a mixed state $\rho$ (a Hermitian positive semidefinite matrix $\rho$ with $\Tr(\rho)=1$) instead of a pure state $\psi$, where we replace the rank $1$ matrix $\psi\psi^*$ by $\rho$ in the above definitions. By convexity this does not change 
the sets $\Cq(\Gamma)$ and $\Cqc(\Gamma)$. It is shown in~\cite{SVW16} that this also does not change the parameter $D_q(P)$, but it is unclear whether or not $D_{qc}(P)$ might decrease. Another variation would be to use projection valued measures (PVMs) instead of POVMs, where the operators are projectors instead of positive semidefinite matrices. This again does not change the sets $\Cq(\Gamma)$ and $\Cqc(\Gamma)$~\cite{NC00}, but the dimension parameters can be larger when restricting to PVMs.

When the two parties have the same question sets ($S=T$) and the same answer sets ($A=B$), a bipartite correlation $P\in \R^\Gamma$ is called \emph{synchronous} if $P(a,b|s,s) = 0$ for all $s$ and $a \neq b$. 
The sets of synchronous (commuting) quantum correlations, denoted $C_{q,s}(\Gamma)$ and $C_{qc,s}(\Gamma)$, are rich enough, so that Connes' embedding conjecture still holds if and only if $\mathrm{cl}(C_{q,s}(\Gamma)) = C_{qc,s}(\Gamma)$ for all $\Gamma$~\cite[Thm.~3.7]{DP16}. The quantum graph parameters discussed in Section~\ref{intro: quantum graph parameters} will be defined through optimization problems over these sets. 

A matrix $M \in \R^{n \times n}$ is \emph{completely positive semidefinite} if there exist $d\in\N$ and Hermitian positive semidefinite matrices $X_1,\ldots,X_n \in \C^{d \times d}$ with $M = (\mathrm{Tr}(X_iX_j))$. The minimal such $d$ is its {\em completely positive semidefinite rank}, denoted $\cpsdrank(M)$.  Completely positive semidefinite matrices are used in~\cite{LP15} to model quantum graph parameters and the completely positive semidefinite rank is investigated in~\cite{PSVW16,GdLL17,PV17,GdLL17a}.
By combining  the proofs from~\cite{SV16} (see also~\cite{MR16}) and~\cite{PSSTW16} one can show the following link between synchronous correlations and completely positive semidefinite matrices.\footnote{See Appendix~\ref{sec:sync} for a proof.}

\begin{restatable*}{proposition}{propcorrelationsynchronous} 
\label{propcorrelation}
The smallest local dimension in which a synchronous quantum correlation $P$ can be realized is given by the completely positive semidefinite rank of the matrix $M_P$ indexed by $S \times A$ with entries $(M_P)_{(s,a),(t,b)} = P(a,b|s,t)$.
\end{restatable*} 

In~\cite{GdLL17a} we use techniques from tracial polynomial optimization to define a semidefinite programming hierarchy $\{\smash{\xi_r^\mathrm{cpsd}}(M)\}$ of lower bounds 
 on $\cpsdrank(M)$. By the above result this hierarchy gives lower bounds on the smallest local dimension in which a synchronous correlation can be realized in the tensor model. However, in~\cite{GdLL17a} we show that the hierarchy typically does not converge to $\cpsdrank(M)$ but instead (under a certain flatness condition) to a parameter $\smash{\xib{cpsd}{*}(M)}$, which can be seen as a block-diagonal version of the completely positive semidefinite rank. 

Here we use similar techniques, now exploiting the special structure of quantum correlations, to construct a hierarchy $\{\xib{q}{r}(P)\}$ of lower bounds on the minimal dimension $D_q(P)$ of any -- not necessarily synchronous -- quantum correlation $P$. The hierarchy converges (under flatness) to a parameter $\xib{q}{*}(P)$, and using the additional structure we can show that $\xib{q}{*}(P)$ is equal to an interesting parameter $A_q(P) \leq D_q(P)$. This parameter describes the minimal average entanglement dimension of a correlation when the parties have free access to shared randomness; see Section~\ref{sec:intro avendi}.

In the rest of the introduction we give a road map through the contents of the paper and state the main results. We will introduce the necessary background along the way.

\subsection{A hierarchy for the average entanglement dimension}\label{sec:intro avendi}

We are interested in the minimal entanglement dimension needed to realize a given correlation $P\in \Cq(\Gamma)$. 
If $P$ is deterministic or only uses local randomness, then $D_{q}(P)=D_{qc}(P)=1$. But other classical correlations (which use shared randomness) have $D_q(P) \geq D_{qc}(P) > 1$, which means the shared quantum state is used as a shared randomness resource. In~\cite{Brunner08} the concept of dimension witness is introduced, where a \emph{$d$-dimensional witness} is defined as a halfspace containing $\mathrm{conv}(\Cq^d(\Gamma))$, but not the full set $\Cq(\Gamma)$. As a measure of entanglement this suggests the parameter
\begin{equation} \label{eqBru}
\mathrm{inf} \Big\{ \max_{i \in [I]} D_q(P_i) :  I \in \N,\, \lambda \in \R_+^I, \, \sum_{i=1}^I \lambda_i = 1,\, P = \sum_{i=1}^I \lambda_i P_i, \, P_i \in \Cq(\Gamma) \Big\}. 
\end{equation}
Observe that, for a bipartite correlation $P$, this parameter is equal to $1$ if and only if $P$ is classical. Hence, it more closely measures the minimal entanglement dimension when the parties have free access to shared randomness. 
From an operational point of view,~\eqref{eqBru} can be interpreted as follows. Before the game starts the parties select a finite number of pure states $\psi_i$ ($i \in I$) (instead of a single one), in possibly different dimensions $d_i$, and POVMs $\{E_s^a(i)\}_a$, $\{F_t^b(i)\}_b$ for each $i \in I$ and $(s,t) \in S \times T$. 
As before, we assume that the parties cannot communicate after receiving their questions $(s,t)$, but now they do have access to shared randomness, which they use to decide on which state $\psi_i$ to use. The parties proceed to measure state $\psi_i$ using POVMs $\{E_s^a(i)\}_a$, $\{F_t^b(i)\}_b$, so that the probability of answers $(a,b)$ is given by the quantum correlation $P_i$. Equation~\eqref{eqBru} then asks for the largest dimension needed in order to generate $P$ when access to shared randomness is free.

It is not clear how to compute~\eqref{eqBru}. Here we propose a variation of~\eqref{eqBru}, and we provide a hierarchy of semidefinite programs that converges to it under flatness.
Instead of considering the largest dimension needed to generate $P$, we consider the \emph{average} dimension. That is,
we minimize $\sum_{i \in I} \lambda_i D_q(P_i)$ over all convex combinations $P = \sum_{i \in I} \lambda_i P_i$.
Hence, the \emph{minimal average entanglement dimension} is given by
\[
A_q(P) = \mathrm{inf} \Big\{ \sum_{i=1}^I \lambda_i D_q(P_i) :  I \in \N,\, \lambda \in \R_+^I, \, \sum_{i=1}^I \lambda_i = 1,\, P = \sum_{i=1}^I \lambda_i P_i, \, P_i \in \Cq(\Gamma) \Big\}
\]
in the tensor model. In the commuting model, $A_{qc}(P)$ is given by the same expression with $D_q(P_i)$ replaced by $D_{qc}(P_i)$. Observe that we need not replace $\Cq(\Gamma)$ by $\Cqc(\Gamma)$ since $D_{qc}(P) = \infty$ for any $P \in \Cqc(\Gamma) \setminus \Cq(\Gamma)$. 

It follows by convexity that for the above definitions it does not matter whether we use pure or mixed states. We show that for the average minimal entanglement dimension it also does not matter whether we use the tensor or commuting model. 

\begin{restatable*}{proposition}{propRelaAvg} \label{prop:relativistic average entanglement dimension}
For any $P\in \Cq(\Gamma)$ we have $A_q(P) = A_{qc}(P)$.
\end{restatable*}

We have $A_q(P) \le D_q(P)$ and $A_{qc}(P) \leq D_{qc}(P)$
 for $P \in \Cq(\Gamma)$, with equality if $P$ is an extreme point of $\Cq(\Gamma)$. Hence, we have $D_q(P) = D_{qc}(P)$ if $P$ is an extreme point of $\Cq(\Gamma)$. We show that the parameter $A_q(P)$ can be used to distinguish between classical and nonclassical correlations.

\begin{restatable*}{proposition}{propeqone} \label{prop:eqone}
For a correlation $P \in \R^\Gamma$ we have $A_q(P) = 1$ if and only if 
$P \in \Cloc(\Gamma)$.
\end{restatable*}

As mentioned before, there exist $\Gamma$ for which $\Cq(\Gamma)$ is not closed~\cite{Slofstra17,DPP17}, which implies the existence of a sequence $\{P_i\} \subseteq \Cq(\Gamma)$ such that $D_q(P) \to \infty$. We show this also implies the existence of such a sequence with $A_q(P_i) \to \infty$.

\begin{restatable*}{proposition}{propuniformbound} \label{propuniformbound}
If $\Cq(\Gamma)$ is not closed, there exists $\{P_i\} \subseteq  \Cq(\Gamma)$ with $A_q(P_i) \to \infty$.
\end{restatable*}

Using tracial polynomial optimization we construct a hierarchy $\{\xib{q}{r}(P)\}$ of lower bounds on $A_{qc}(P)$. For each $r \in \N$ this is a  semidefinite program, and for $r = \infty$ it is an infinite dimensional semidefinite program. We further define a (hyperfinite) variation $\xib{q}{*}(P)$ of $\xib{q}{\infty}(P)$ by adding a finite rank constraint, so that
\[
\xib{q}{1}(P) \leq \xib{q}{2}(P) \leq \ldots \leq \xib{q}{\infty}(P) \leq \xib{q}{*}(P) \leq A_{qc}(P).
\]
We do not know whether $\xib{q}{\infty}(P) = \xib{q}{*}(P)$ always holds; this question is related to Connes' embedding conjecture~\cite{KS08}.
First we show that we imposed enough constraints in the bounds $\xib{q}{r}(P)$ so that $\xib{q}{*}(P)=A_{qc}(P)$.
\begin{restatable*}{proposition}{propstarisaq}
For any $P\in \Cq(\Gamma)$ we have $\xib{q}{*}(P) = A_\mathrm{qc}(P)$.
\end{restatable*}

Then we show that the infinite dimensional semidefinite program $\xib{q}{\infty}(P)$ is the limit of the finite dimensional semidefinite programs.
\begin{restatable*}{proposition}{propconvergetinfty}
For any $P\in \Cq(\Gamma)$ we have $\xib{q}{r}(P)\to \xib{q}{\infty}(P)$ as $r\to \infty$.
\end{restatable*}

Finally we give a criterion under which finite convergence $\xib{q}{r}(P) = \xib{q}{*}(P)$ holds. The definition of flatness follows later in the paper; here we only note that it is an easy to check criterion given the output of the semidefinite programming solver.
\begin{restatable*}{proposition}{propqflat}
If $\xib{q}{r}(P)$ admits a $(\lceil r/3 \rceil+1)$-flat optimal solution, $\xib{q}{r}(P) = \xib{q}{*}(P)$.
\end{restatable*}

\subsection{Quantum graph parameters}\label{intro: quantum graph parameters}

Nonlocal games have been introduced in quantum information theory as abstract models  to quantify the power of entanglement, in particular,  in how much the sets $\Cq(\Gamma)$ and $\Cqc(\Gamma)$ differ from $\Cloc(\Gamma)$. A \emph{nonlocal game} is defined by a  probability distribution $\pi \colon S \times T \to [0,1]$ and a predicate $f \colon A \times B \times S \times T \to \{0,1\}$. Alice and Bob receive a question pair $(s,t)\in S\times T$ with probability $\pi(s,t)$.  They know the game parameters $\pi$ and $f$, but they do not know each other's questions, and they cannot communicate after they receive their questions.  
Their answers $(a,b)$ are determined according to some  correlation $P\in \R^\Gamma$,  called their {\em strategy}, on which they may agree before the start of the game, and which can be classical or quantum depending on whether $P$ belongs to $\Cloc(\Gamma)$,  $\Cq(\Gamma)$, or $\Cqc(\Gamma)$.  Then their corresponding winning probability is given by
\begin{equation}\label{eqvalue}
\sum_{(s,t)\in S\times T} \pi(s,t)\sum_{(a,b)\in A\times B} P(a,b|s,t) f(a,b,s,t).
\end{equation}
A strategy $P$ is called  {\em perfect} if the above winning probability is equal to one, that is, if for all $(a,b,s,t)\in \Gamma$ we have
$$\big(\pi(s,t)>0 \quad \text{ and } \quad f(a,b,s,t)=0 \big) \quad \Longrightarrow \quad P(a,b|s,t)=0.$$

Computing the maximum winning probability of a nonlocal game is an instance of linear optimization over $\Cloc(\Gamma)$ in the classical setting, and over $\Cq(\Gamma)$ or $\Cqc(\Gamma)$ in the quantum setting. 
Since the inclusion $\Cloc(\Gamma) \subseteq \Cq(\Gamma)$ can be strict, the winning probability can be higher when the parties have access to entanglement. In fact there are nonlocal games that can be won with probability $1$ by using entanglement, but with probability strictly less than $1$ in the classical setting. 

The quantum graph parameters are analogues of the classical parameters defined through the coloring and stability number games as described below. These nonlocal games use the set $[k]$ (whose elements are denoted as $a,b$) and the set $V$ of vertices of $G$ (whose elements are denoted as $i,j$) as question and answer sets.

In the {\em quantum coloring game}, introduced in~\cite{AHKS06,CMNSW07}, we have a graph $G = (V,E)$ and an integer $k$. Here we have question sets $S=T=V$ and answer sets $A=B=[k]$, and the distribution $\pi$ is strictly positive on $V \times V$. The predicate $f$ is such that the players' answers have to be consistent with having a $k$-coloring of $G$; that is, $f(a,b,i,j)=0$ precisely when ($i=j$ and $a\ne b$) or ($\{i,j\}\in E$ and $a=b$). This expresses the fact that if Alice and Bob receive the same vertex they should return the same color and if they receive adjacent vertices they should return distinct colors.
A perfect classical strategy exists if and only if a perfect deterministic strategy exists, and a perfect deterministic strategy corresponds to a $k$-coloring of $G$. Hence the smallest number $k$ of colors for which there exists a perfect classical strategy is equal to the classical chromatic number $\chi(G)$. 
It is therefore natural to define the quantum chromatic number as the smallest $k$ for which there exists a perfect quantum strategy. Since such a strategy is necessarily synchronous we get the following definition.

\begin{definition}\label{defchiq}
The (commuting) quantum chromatic number $\chi_q(G)$ (resp., $\chi_{qc}(G)$) is the smallest integer $k\in\N$ for which there exists a synchronous correlation $P=(P(a,b|i,j))$ in $ C_{q,s}([k]^2\times V^2)$ (resp.,  $C_{qc,s}([k]^2\times V^2)$) such that 
\begin{align*}
P(a,a|i,j)=0 & \quad  \text{for all} \quad a\in [k], \{i,j\}\in E.
\end{align*}
\end{definition}

In the {\em quantum stability number game}, introduced in~\cite{MR16,Rob13}, we again have a graph $G = (V,E)$ and $k \in \N$, but now we use the  question set $[k]\times [k]$  and the answer set $V\times V$. The distribution $\pi$ is again strictly positive on the question set
and now the predicate $f$ of the game is such that the players' answers have to be consistent with having a stable set of size $k$, that is, 
$f(i,j,a,b)=0$ precisely when ($a=b$ and $i\ne j$) or ($a\ne b$ and ($i=j$ or $\{i,j\}\in E$)). This expresses the fact that if Alice and Bob receive the same index $a=b\in [k]$ they should answer with the same vertex $i=j$ of $G$, and if they receive distinct indices $a\ne b$ from $[k]$ they should answer with distinct nonadjacent vertices $i$ and $j$ of $G$. There is a perfect classical strategy precisely when there exists a stable set of size $k$, so that the largest integer $k$ for which there exists a  perfect classical strategy is equal to the stability number $\alpha (G)$. Again, such a strategy is necessarily synchronous, so we get the following definition.

\begin{definition}\label{defaq}
The (commuting) stability number $\alpha_q(G)$ (resp., $\alpha_{qc}(G)$) is the largest  integer $k\in \N$ for which there exists a synchronous correlation $P=(P(i,j|a,b))$ in
$C_{q,s}(V^2\times [k]^2)$ (resp.,  $C_{qc,s}(V^2\times [k]^2)$) such that 
\begin{align*}
P(i,j|a,b)=0 & \quad \text{whenever} \quad  (i=j \text{ or } \{i,j\}\in E) \text{ and } a\ne b\in [k].
\end{align*}
\end{definition}

The classical parameters $\chi(G)$ and $\alpha(G)$ are NP-hard.  The same holds for the quantum coloring number  $\chi_q(G)$~\cite{Ji13}
and also for the quantum stability number $\alpha_q(G)$ in view of the following reduction to coloring shown in~\cite{MR16}:
\begin{equation}\label{eqcolstab}
\chi_q(G)=\min\{k\in \N: \alpha_q(G\Box K_k)=|V|\}.
\end{equation}
Here $G\Box K_k$ is the Cartesian product of the graph $G=(V,E)$ and the complete graph $K_k$. 
By construction we have $\chi_{qc}(G)\le \chi_q(G)\le \chi(G)$ and $\alpha(G)\le \alpha_q(G)\le \alpha_{qc}(G)$.
The separations between $\chi_q(G)$ and $\chi(G)$, and  between  $\alpha_q(G)$ and $\alpha(G)$, can  be exponentially large in the number of vertices; this is the case for the graphs with vertex set $\{\pm 1\}^n$ for $n$ a multiple of $4$, where two vertices are adjacent if they are orthogonal~\cite{AHKS06,MR16,MSS13}. 
While it was recently shown that the sets $C_{q,s}(\Gamma)$ and $C_{qc,s}(\Gamma)$ can be different, it is not known whether there is a separation  between the parameters $\chi_q(G)$ and $\chi_{qc}(G)$, and between $\alpha_q(G)$ and $\alpha_{qc}(G)$.

\bigskip

We now  give an overview of the results of Section~\ref{sec:estquantumchrom} and refer to that section for formal definitions. We first reformulate the quantum graph parameters in terms of $C^*$-algebras, which allows us to use techniques from tracial polynomial optimization to formulate bounds on the quantum graph parameters. We define a hierarchy $\{\gamma_r^\mathrm{col}(G)\}$ of lower bounds on the commuting quantum chromatic number and a hierarchy $\{\gamma_r^\mathrm{stab}(G)\}$ of upper bounds on the commuting quantum stability number. We show the following convergence results for these hierarchies. 

\begin{restatable*}{proposition}{LemConvergenceQuantum}
There is an $r_0 \in \N$ such that $\gamma_r^\mathrm{col}(G) = \chi_{qc}(G)$ and $\gamma_r^\mathrm{stab}(G) = \alpha_{qc}(G)$ for all $r \geq r_0$. Moreover, if $\gamma_r^\mathrm{col}(G)$ admits a flat optimal solution, then $\gamma_r^\mathrm{col}(G) = \chi_q(G)$, and if $\gamma_r^\mathrm{stab}(G)$ admits a flat optimal solution, then $\gamma_r^\mathrm{stab}(G) = \alpha_q(G)$. 
\end{restatable*}

Then we define tracial analogues $\{\xib{stab}{r}(G)\}$ and $\{\xib{col}{r}(G)\}$  of Lasserre type bounds on $\alpha(G)$ and $\chi(G)$ that provide hierarchies of bounds for their quantum analogues. These bounds are more economical than the bounds $\gamma^{\rm col}_r(G)$ and $\gamma^{\rm stab}_r(G)$ (since they use less variables) and also permit to recover some known bounds for the quantum parameters. We show  that $\xib{stab}{*}(G)$, which is the parameter $\xib{stab}{\infty}(G)$ with an additional rank constraint on the matrix variable, coincides with the projective packing number $\alpha_p(G)$ from~\cite{Rob13} and that $\xib{stab}{\infty}(G)$ upper bounds $\alpha_{qc}(G)$.

\begin{restatable*}{proposition}{lemalphap}\label{lemalphap}
We have $\xib{stab}{*}(G) = \alpha_p(G)\ge \alpha_q(G)$ and $\xib{stab}{\infty}(G)\ge \alpha_{qc}(G)$.
\end{restatable*}

Next, we consider the chromatic number. The tracial hierarchy $\{\xib{col}{r}(G)\}$  unifies two known bounds: the projective rank $\xi_f(G)$, a lower bound on the quantum chromatic number from~\cite{MR16}, and the tracial rank $\xi_{tr}(G)$, a lower bound on the commuting quantum chromatic number from~\cite{PSSTW16}. 
In~\cite[Cor. 3.10]{DP16} it is shown that the projective rank and the tracial rank coincide if Connes' embedding conjecture is true. 

\begin{restatable*}{proposition}{lemchif}\label{lemchif}
We have $\xib{col}{*}(G) = \xi_f(G)\le \chi_q(G)$ and $\xib{col}{\infty}(G)=\xi_{tr}(G)\le \chi_{qc}(G)$.
\end{restatable*}

We compare the hierarchies $\xib{col}{r}(G)$ and $\gamma_r^\mathrm{col}(G)$, and the hierarchies $\xib{stab}{r}(G)$ and $\gamma^{\rm stab}_r(G)$.
For the coloring parameters, we show the analogue of  reduction~\eqref{eqcolstab}. 

\begin{restatable*}{proposition}{propXicol}\label{propXicol}
For  $r\in \N\cup\{\infty\}$ we have $\gamma^{\rm col}_r(G)= 
\min\{k: \xib{stab}{r}(G\Box K_k)=|V|\}.$
\end{restatable*}

We show an analogous statement for the stability parameters, when using the homomorphic graph product of $K_k$ with the complement of $G$, denoted here as $K_k\star G$, and the following reduction shown in~\cite{MR16}:
\[
\alpha_q(G)=\max\{k\in\N: \alpha_q(K_k\star G)=k\}.
\]

\begin{restatable*}{proposition}{propXistab}\label{propXistab}
For $r \in \N\cup\{\infty\}$ we have $\gamma^{\rm stab}_r(G)= 
\max\{k: \xib{stab}{r}(K_k\star G)=k\}.$
\end{restatable*}

Finally, we show that the hierarchies $\{\gamma^{\rm col}_r(G)\}$ and $\{\gamma^{\rm stab}_r(G)\}$ refine the hierarchies $\{\xib{col}{r}(G)\}$ and $\{\xib{stab}{r}(G)\}$.

\begin{restatable*}{proposition}{propcolstabcompare}
For $r\in \N\cup\{\infty, *\}$, $\xib{col}{r}(G) \le \gamma^{\rm col}_r(G)$ and $ \xib{stab}{r}(G)\ge \gamma^{\rm stab}_r(G)$.
\end{restatable*}

\subsection{Techniques from noncommutative polynomial optimization}

To derive our bounds we use techniques from tracial polynomial optimization. This  is a noncommutative extension of the widely used moment and sum-of-squares techniques from Lasserre~\cite{Las01} and Parrilo~\cite{Par00} in polynomial optimization, dealing with the problem of minimizing a multivariate polynomial  
over a feasible region defined by polynomial inequalities.
These techniques have been adapted to the noncommutative setting in~\cite{NPA08} and~\cite{DLTW08} for approximating the set $\Cqc(\Gamma)$ of commuting quantum correlations and the winning probability of nonlocal games over $\Cqc(\Gamma)$ (and, more generally, computing Bell inequality violations).
In~\cite{NPA10,NPA12} this approach  has been extended to the general eigenvalue optimization problem, of the form 
\begin{align*}
\mathrm{inf} \big\{ \psi^* f(X_1,\ldots,X_n) \psi : \;& d \in \N,\, \psi \in \C^d \text{ unit vector},\, X_1,\ldots,X_n \in \C^{d \times d},\\
& g(X_1,\ldots,X_n) \succeq 0 \text{ for } g \in \mathcal G\big\}.
\end{align*}
Here, the matrix variables $X_i$ have free dimension $d\in \N$ and $\{f\}\cup  \mathcal G\subseteq \R\langle x_1,\ldots,x_n\rangle$ is a set of symmetric polynomials in noncommutative variables. In tracial optimization, instead of minimizing the smallest eigenvalue of $f(X_1,\ldots,X_n)$, we minimize its normalized trace ${\rm Tr}(f(X_1,\ldots,X_n))/d$ (so that the identity matrix has trace one)~\cite{BK12,BCKP13,BKP16,KP16}. 
The moment approach for these problems relies on minimizing $L(f)$, where $L$ is  a linear functional on the space of noncommutative polynomials satisfying some necessary conditions, and $L(f)$ models $\psi^*f(X_1,\ldots,X_n)\psi$ or ${\rm Tr}(f(X_1,\ldots,X_n))/d$. By  truncating the degrees  one gets hierarchies of lower bounds for the original problem. The asymptotic limit of these bounds involves operators $X_i$ on a Hilbert space (possibly of infinite dimension). In tracial optimization this leads to allowing solutions $X_i$ in a $C^*$-algebra $\mathcal A$ equipped with a tracial state $\tau$, where $\tau(f(X_1,\ldots,X_n))$ is minimized.

An important feature in noncommutative optimization  is the dimension independence: the optimization is over all possible matrix sizes $d \in \mathbb N$. 
In some applications one may want to restrict to optimizing over matrices with restricted size $d$. In~\cite{NMVT15,NFAV15} techniques are developed that allow to incorporate this dimension restriction by suitably selecting the linear functionals $L$ in a specified space; this is used to give bounds on the maximum violation of a Bell inequality 
in a fixed dimension.
A related natural problem is to decide what is the minimum dimension $d$ needed to realize a given algebraically defined object, such as a (commuting) quantum correlation $P$. 
We propose an approach based on tracial optimization: starting from the observation that the trace of the $d\times d$ identity matrix gives its size $d$, we consider the problem of minimizing $L(1)$ where $L$ is a linear functional modeling the non-normalized matrix trace. This approach has been used in several recent works~\cite{TS15,Nie16,GdLL17a} for lower bounding factorization ranks of matrices and tensors. 

\section{A hierarchy for the minimal entanglement dimension} \label{sec: average entanglement dimension}

\subsection{The minimal average entanglement dimension}

We start by showing that it does not matter whether we use the tensor or the commuting model when defining the average entanglement dimension.

\propRelaAvg
\begin{proof}
The easy inequality $A_{qc}(P) \leq A_q(P)$ follows from $E_s^a \otimes F_t^b = (E_s^a \otimes I) (I \otimes F_t^b)$.

For the other inequality we suppose $P = \smash{\sum_{i=1}^I} \lambda_i P_i$ is feasible for $A_{qc}(P)$. This means we have POVMs $\{X_s^a(i)\}_a$ and $\{Y_t^b(i)\}_b$ in $\C^{d_i \times d_i}$ with $[X_s^a(i),Y_t^b(i)] = 0$ and unit vectors $\psi_i\in\C^{d_i}$ such that $P_i(a,b|s,t) = \psi_i^* X_s^a(i) Y_t^b(i) \psi_i$ for all $(a,b,s,t) \in \Gamma$ and $i\in [I]$. We will construct a feasible solution to $A_q(P)$ with value at most $\sum_i\lambda_id_i$. 

Fix some index $i\in [I]$. By Artin-Wedderburn theory applied to $\C\langle \{X^a_s(i)\}_{a,s}\rangle$,
the $*$-algebra generated by the matrices $X^a_s(i)$ for $(a,s)\in A\times S$, there exists a unitary matrix $U_i$ and integers $K_i,m_k,n_k$ such that
\[
U_i \C \langle \{X_s^a(i)\}_{a,s} \rangle U_i^* = \bigoplus_{k=1}^{K_i} (\C^{n_k \times n_k} \otimes I_{m_k}) \quad \text{and} \quad d_i=\sum_{k=1}^{K_i} m_kn_k.
\]
By the commutation relations each matrix $Y^b_t(i)$ commutes with all the matrices in $\C\langle \{X^a_s(i)\}_{a,s}\rangle$, and thus $U_i Y_t^b(i)U_i^*$ lies in the algebra $\bigoplus_k (I_{n_k} \otimes \C^{m_k \times m_k})$. Hence, we may assume
\[
X_s^a(i) = \bigoplus_{k=1}^{K_i} E_s^a(i,k) \otimes I_{m_k},
\quad Y_t^b(i) = \bigoplus_{k=1}^{K_i} I_{n_k} \otimes F_t^b(i,k), \quad 
\psi_i = \bigoplus_{k=1}^{K_i} \psi_{i,k},
\]
with $E^a_s(i,k)\in \C^{n_k\times n_k}$, $F^b_t(i,k)\in\C^{m_k\times m_k}$, and $\psi_{i,k} \in \C^{m_kn_k}$. Then we have
\begin{align*}
P_i(a,b|s,t)=\mathrm{Tr}(X_s^a(i) Y_t^b(i) \psi_i\psi_i^*)
&=\sum_{k=1}^{K_i} \|\psi_{i,k}\|^2\, \underbrace{\mathrm{Tr}\left(E_s^a(i,k) \otimes F_t^b(i,k) \frac{\psi_{i,k} \psi_{i,k}^*}{\|\psi_{i,k}\|^2}\right)}_{Q_{i,k}(a,b|s,t)},
\end{align*}
where $Q_{i,k}\in \Cq(\Gamma)$.
As $\sum_k \|\psi_{i,k}\|^2 = \|\psi_i\|^2 = 1$, we have that $P_i=\sum_k \|\psi_{i,k}\|^2 Q_{i,k}$ is a convex combination of the $Q_{i,k}$'s.

We now show that $Q_{i,k}\in \Cq^{\min\{m_k,n_k\}}(\Gamma)$.
Consider the Schmidt decomposition
$
\psi_{i,k}/\|\psi_{i,k}\| = \sum_{l=1}^{\min\{m_k, n_k\}} \lambda_{i,k,l} \, v_{i,k,l} \otimes w_{i,k,l},
$
where $\lambda_{i,k,l}\ge 0$ and $\{v_{i,k,l}\}_{l=1}^{n_k} \subseteq \C^{n_k}$ and $\{w_{i,k,l}\}_{l=1}^{m_k} \subseteq \C^{m_k}$ are orthonormal bases. Define unitary matrices $V_k\in\C^{n_k\times n_k}$ and $W_k\in\C^{m_k\times m_k}$ such that $V_k v_{i,k,l}$ is the $l$th unit vector in $\R^{n_k}$ and $W_k w_{i,k,l}$ is the $l$th unit vector in $\R^{m_k}$ for $l\le \min\{m_k,n_k\}$. Let $E_s^a(i,k)'$ (resp., $F_t^b(i,k)'$) be the leading principal submatrices of $V_k E_s^a(i,k) V_k^*$ (resp., $W_k F_t^b(i,k) W_k^*$) of size $\min\{m_k, n_k\}$. Moreover, set $\phi_{i,k} = \sum_{l=1}^{\min\{m_k,n_k\}} \lambda_{i,k,l} e_l \otimes e_l$, where $e_l$ is the $l$th unit vector in $\R^{\min\{m_k, n_k\}}$.
Then we have
\begin{align*}
Q_{i,k}(a,b|s,t) 
&=
\mathrm{Tr}\left(E_s^a(i,k) \otimes F_t^b(i,k) \frac{\psi_{i,k} \psi_{i,k}^*}{\|\psi_{i,k}\|^2}\right)\\
& = \sum_{l,l'=1}^{\min\{m_k, n_k\}} \lambda_{i,k,l} \lambda_{i,k,l'} v_{i,k,l}^* E_s^a(i,k) v_{i,k,l'} w_{i,k,l}^* F_t^b(i,k) w_{i,k,l'}\\
& = \sum_{l,l'=1}^{\min\{m_k, n_k\}} \lambda_{i,k,l} \lambda_{i,k,l'} e_l^* E_s^a(i,k)' e_{l'} e_l^* F_t^b(i,k)' e_{l'}\\
&= \mathrm{Tr}((E_s^a(i,k)' \otimes F_t^b(i,k)') \phi_{i,k} \phi_{i,k}^*),
\end{align*}
thus showing $Q_{i,k}\in \Cq^{\min\{m_k,n_k\}}(\Gamma)$. Since $P=\sum_{i,k} \lambda_i \|\psi_{i,k}\|^2 Q_{i,k}$ is a convex decomposition, we obtain
\[
A_q(P)\le \sum_{i,k} \lambda_i \|\psi_{i,k}\|^2 \min\{m_k,n_k\}^2 \leq  \sum_{i,k} \lambda_i \min\{m_k,n_k\}^2 \le \sum_{i,k} \lambda_i m_kn_k= \sum_i \lambda_i d_i. \ \qedhere
\]
\end{proof}

We now show  
the parameter $A_q(\cdot)$ permits to characterize classical correlations.
\propeqone
\begin{proof}
If $P \in \Cloc(\Gamma)$, then $P$ can be written as a convex combination of deterministic correlations (which are contained in $\Cq^1(\Gamma)$), hence $A_q(P) =1$.

On the other hand, if $A_q(P) = 1$, then there exist convex decompositions indexed by $l \in \N$: $P = \sum_{i \in I^l} \lambda_i^l P_i^{l}$ with $\{P_i^l\} \subseteq \Cq(\Gamma)$ and $\lim_{l \to \infty} \sum_{i\in I^l} \lambda_l D_q(P_i^l) = 1$.
Decompose $I^l$ as the disjoint union $I_-^l \cup I_+^l$ so that $D_q(P_i)$ is equal to $1$ for $i \in I_-^l$ and strictly greater than $1$ for $i \in I_+^l$. Let $\varepsilon > 0$. For all $l$ sufficiently large we have
\[
\big(1-\sum_{i \in I_+^l} \lambda_i^l\Big) + 2 \sum_{i \in I_+^l} \lambda_i^l \leq \sum_{i \in I_-^l} \lambda_i^l + \sum_{i \in I_+^l} \lambda_i^l D_q(P_i^l) \leq 1 +\varepsilon,
\]
which shows that $\smash{\sum_{i \in I_+^l}} \lambda_i^l \leq \varepsilon$. This shows that $P$ is the limit of convex combinations of deterministic correlations, which implies that $P \in \Cloc(\Gamma)$.
\end{proof}

\propuniformbound
\begin{proof}
Assume for contradiction  
there exists  an integer $K$ such that $A_q(P) \le K$ for all $P \in \Cq(\Gamma)$; we 
show this results in a uniform upper bound $K'$ on $D_{qc}(P)$, 
which implies $\Cq(\Gamma)=C_{qc}^{K'}(\Gamma)$ is closed. 
For this,  we will first show that $P\in \mathrm{conv}(C_{qc}^K(\Gamma))$.

In a first step
observe that any $P\in \Cq(\Gamma)\setminus \mathrm{conv}(C_{qc}^K(\Gamma))$ can be decomposed as
\begin{equation}\label{eqPRQ}
P=\mu_1 R_1+ (1-\mu_1)Q_1,\end{equation}
where $R_1\in \Cq(\Gamma)$, $Q_1\in \text{conv} (C_{qc}^K(\Gamma))$, and $0< \mu_1\le K/(K+1)$.
Indeed, by assumption and using Proposition~\ref{prop:relativistic average entanglement dimension},  $A_{qc}(P)=A_q(P)\le K$, so $P$ can be written as a convex combination
$P = \sum_{i \in I} \lambda_i P_i$ with $\{P_i\} \subseteq  \Cq(\Gamma)$ and $\sum_{i \in I} \lambda_i D_{qc}(P_i) \leq K$.
As $P\not\in \mathrm{conv}(C_{qc}^K(\Gamma))$, the set $J$ of indices $i\in I$ with $D_{qc}(P_i)\ge K+1$ is non empty. Then   
$(K+1) \sum_{i \in J} \lambda_i \leq \sum_{i \in J} \lambda_i D_{qc}(P_i) \leq K$,
and thus $0<\mu_1:=\sum_{i\in I_+}\lambda_i \leq K/(K+1)$. Hence~\eqref{eqPRQ} holds after setting $R_1=(\sum_{i\in J}\lambda_i P_i)/\mu_1$ and
$Q_1=(\sum_{i\in I\setminus J}\lambda_i P_i)/(1-\mu_1)$.

As $R_1\in C_q(\Gamma)\setminus C_{qc}^K(\Gamma)$, we may repeat  the same argument for $R_1$. By iterating we obtain for each integer $k\in \N$ a decomposition 
\[
P=\mu_1\mu_2\cdots \mu_k R_k+  \underbrace{(1-\mu_1)Q_1+\mu_1(1-\mu_2)Q_2+\ldots + \mu_1\mu_2\cdots \mu_{k-1}(1-\mu_k)Q_k}_{= (1-\mu_1\mu_2\cdots \mu_k)\hat Q_k},
\]
where $R_k\in \Cq(\Gamma)$, $\hat Q_k\in \text{conv}(C_{qc}^K(\Gamma))$ and $\mu_1\mu_2\cdots \mu_k\le (K/(K+1))^k$.
As the entries of $R_k$ lie in $[0,1]$ we can conclude that $\mu_1\mu_2\cdots \mu_k R_k$ tends to 0 as $k\to \infty$.
Hence the sequence $(\hat Q_k)_k$ has a limit $\hat Q$ and $P=\hat Q$ holds. As all $\hat Q_k$ lie in the compact set   $\text{conv}(C_{qc}^K(\Gamma))$, we also have $P\in \text{conv}(C_{qc}^K(\Gamma))$. 

The extreme points of the compact convex set $\mathrm{conv}(\smash{C_{qc}^K(\Gamma)})$ lie in $\smash{C_{qc}^K(\Gamma)}$, so, by the Carath\'eodory theorem,  any $P\in \mathrm{conv}(\smash{C_{qc}^K(\Gamma)})$ is a convex combination of  $c$ elements from $\smash{C_{qc}^K(\Gamma)}$, where $c = |\Gamma|+1-|S||T|$. 
By using a direct sum construction one can obtain  $D_{qc}(P) \leq cK$, which shows $K':=cK$ is a uniform upper bound on $D_{qc}(P)$ for all $P\in C_q(\Gamma)$.
\end{proof}

\subsection{Setup of the hierarchy}

We will now construct a hierarchy of lower bounds on the minimal entanglement dimension, using its formulation via $A_{qc}(P)$. Our approach is based on noncommutative polynomial optimization, thus similar to the approach in~\cite{GdLL17a} for bounding matrix factorization ranks. 

We first need some notation. Set $\bx=\big\{x^a_s: (a,s)\in A\times S\big\}$ and $\by=\big\{y^b_t: (b,t)\in B\times T\big\}$,
and let $\langle \bx,\by,z\rangle_r$ be the set of all words of length at most $r$ in the $n = |S||A|+|T||B| + 1$ symbols $x_s^a$, $y_t^b$, and $z$. Moreover, set $\langle \bx, \by, z \rangle = \langle \bx, \by, z \rangle_\infty$. We equip $\langle \bx, \by, z \rangle_r$ with an involution $w \mapsto w^*$ that reverses the order of the symbols in the words and leaves the symbols $x^a_s,y^b_t,z$ invariant; e.g., $(x_s^az)^* = z x_s^a$. Let $\R\langle \bx, \by, z \rangle_r$ be the vector space of all real linear combinations of the words of length (aka degree) at most $r$. The space $\R\langle \bx, \by, z \rangle = \R\langle \bx, \by, z \rangle_\infty$ is the $*$-algebra with Hermitian generators $\{x_s^a\}$, $\{y_t^b\}$, and $z$, and the elements in this algebra are called  \emph{noncommutative polynomials} in the variables $\{x^a_s\},\{y^b_t\},z$.

The hierarchy is based on the following idea: For any feasible solution to $A_{qc}(P)$, its objective value can be modeled as $L(1)$ for a certain tracial linear form $L$ on the space of noncommutative polynomials (truncated to degree $2r$).

Indeed, assume  $\{(P_i, \lambda_i)_i\}$ is a feasible solution to the program $A_{qc}(P)$ defined in Section~\ref{sec:intro avendi}, where $P_i(a,b|s,t) = \mathrm{Tr}\big(X_s^a(i) Y_t^b(i) \psi_i \psi_i^*\big)$ with $X_s^a(i), Y_t^b(i) \in \C^{d_i \times d_i}$, $\psi_i \in \C^{d_i}$, and $d_i = D_{qc}(P_i)$. For $r\in \N \cup \{\infty\}$, consider the linear functional $L \in \R\langle \bx, \by, z \rangle_{2r}^*$ defined by
\[
L(p) = \smash{\sum_i} \lambda_i \, \mathrm{Re}(\mathrm{Tr}(p(\mathbf X(i), \mathbf Y(i), \psi_i \psi_i^*)))\quad \text{ for } \quad p\in \R\langle \bx,\by,z\rangle_{2r}.
\]
Here, for each index $i$, we set $\mathbf X(i)=(X^a_s(i): (a,s)\in A\times S)$, $\mathbf Y(i)=(Y^b_t(i): (b,t)\in B\times T)$, and replace the variables $x^a_s$, $y^b_t$, $z$ by $X^a_s(i)$, $Y^b_t(i)$, and $\psi_i\psi_i^*$.
Then  $L(1) = \sum_i \lambda_i d_i$. That is, $L(1)$ is the objective value of the feasible solution  $\{(P_i, \lambda_i)_i\}$ to $A_{qc}(P)$. 
We will  
identify several computationally tractable properties that this 
$L$ satisfies. Then the  hierarchy of lower bounds on $A_{qc}(P)$ consists of optimization problems where we minimize $L(1)$ over the set of linear functionals that satisfy these properties. 

 First note that $L$ is \emph{symmetric}, that is, $L(w) = L(w^*)$ for all $w \in \langle \bx, \by, z\rangle_{2r}$, and \emph{tracial}, that is, $L(ww') = L(w'w)$ for all $w,w' \in \langle \bx, \by, z\rangle$ with $\deg(ww')\le 2r$.

For all $p \in \R\langle \bx, \by, z\rangle_{r-1}$ we have 
\[
L(p^*x_s^ap) = \sum_i \lambda_i \, \mathrm{Re}(\mathrm{Tr}(M(i)^* X_s^a(i) M(i)) \geq 0, \text{ where } M(i) = p(\mathbf X(i), \mathbf Y(i), \psi_i \psi_i^*),
\]
as $M(i)^* X_s^a(i) M(i)$ is  positive semidefinite since $X^a_s(i)$ is positive semidefinite. In the same way we have $L(p^*y_t^bp) \geq 0$ and $L(p^*z p) \geq 0$. That is, if we set
\[
\mathcal G = \big\{x_s^a : s \in S, \, a \in A\big\} \cup \big\{y_t^b : t \in T, \, b \in B\big\}\cup\{z\},
\]
then $L$ is nonnegative (denoted as $L \geq 0$) on the \emph{truncated quadratic module}
\begin{equation}\label{eq:quadratic module}
\MM_{2r}(\mathcal G)=\mathrm{cone}\Big\{p^*g p: p\in \R\langle \bx, \by, z\rangle, \ g\in \mathcal G\cup\{1\},\ \deg(p^*g p)\le 2r\Big\}.
\end{equation}
Similarly,  setting 
\[
\mathcal H = \big\{z - z^2\big\} \cup \big\{1 - \sum_{a\in A} x_s^a : s \in S\big\} \cup \big\{1 - \sum_{b\in B} y_t^b : t \in T\big\}\cup \big\{[x_s^a, y_t^b] : (s,t,a,b) \in \Gamma\big\},
\]
we have $L = 0$ on the \emph{truncated ideal}
\begin{equation}
\mathcal I_{2r}(\mathcal H) = \Big\{ph : p\in \R\langle \bx, \by, z\rangle, \ h\in \mathcal H,\ \deg(ph)\le 2r\Big\}.
\end{equation}
Moreover, we have $L(z) = \sum_i\lambda_i \mathrm{Re}(\Tr(\psi_i\psi_i^*))=1$. In addition, for any matrices $U,V \in \C^{d_i \times d_i}$ we have 
\[
\psi_i \psi_i^* U \psi_i \psi_i^* V \psi_i \psi_i^* = \psi_i \psi_i^* V \psi_i \psi_i^* U \psi_i \psi_i^*, 
\] and therefore, in particular,
\[
L(w z u z v z) = L(wz v z u z) \quad \text{for all} \quad u,v,w \in \langle  \bx, \by,z\rangle \quad \text{with} \quad \deg(w z u z v z ) \leq 2r.
\]
That is, we have $L = 0$ on $\mathcal I_{2r}(\mathcal R_r)$, where
\[
\mathcal R_r = \big\{z u z v z - z v z u z : u,v \in u,v \in \langle \bx, \by,z\rangle \text{ with } \deg(z u z v z) \leq 2r\big\}.
\]
We get the idea of adding these last constraints from~\cite{NPA12}, where this is used to study the mutually unbiased bases problem.

We call $\MM(\mathcal G) = \MM_\infty(\mathcal G)$  the quadratic module generated by $\mathcal G$, and we call $\MI(\mathcal H \cup \mathcal R_\infty) = \MI_\infty(\mathcal H \cup \mathcal R_\infty)$ the ideal generated by $\mathcal H\cup \mathcal R_\infty$. 
 
 For $r \in \N \cup \{\infty\}$ we can now define the parameter:
\begin{align*}
\xib{q}{r}(P) = \mathrm{min} \Big\{ L(1) : \; & L \in \R\langle \bx, \by, z \rangle_{2r}^* \text{ tracial and symmetric},\\[-0.15em]
& L(z) = 1, \, L(x_s^a y_t^b z) = P(a,b|s,t) \text{ for all } (a,b,s,t)\in\Gamma,\\[-0.1em]
&L \geq 0 \text{ on } \mathcal M_{2r}(\mathcal G),\, L = 0 \text{ on } \mathcal I_{2r}(\mathcal H \cup \mathcal R_r)\Big\}.
\end{align*}
Additionally, we define $\xib{q}{*}(P)$ by adding the constraint $\rank(M(L)) < \infty$ to $\xib{q}{\infty}(P)$.
By construction this gives a hierarchy of lower bounds for $A_{qc}(P)$:
$$\xib{q}{1}(P)\le \ldots \le \xib{q}{r}(P) \le \xib{q}{\infty}(P)\le \xib{q}{*}(P) \le A_{qc}(P).$$
Note that for order $r=1$ we get the trivial lower bound $\xib{q}{1}(P)=1$.

For each finite $r\in \N$ the parameter $\xib{q}{r}(P)$ can be computed by  semidefinite programming. Indeed, the condition $L\ge 0$ on $\mathcal M_{2r}(\mathcal G)$ means that 
$L(p^*gp)\ge 0$ for all  $g \in \mathcal G \cup \{1\}$ and all  polynomials $p\in \R\langle \bx,\by,z\rangle$ with degree at most ${r- \lceil \deg(g)/2\rceil}$. This
 is equivalent to requiring that the matrices $(L(w^*gw'))$, indexed by all words $w,w'$ with degree at most ${r- \lceil \deg(g)/2\rceil}$, 
  are positive semidefinite. To see this, write $p = \sum_{w} 
p_w w$ and let $\hat p=(p_w)$ denote the vector of coefficients, 
then $L(p^*gp) \geq 0$ is equivalent to $\hat p^\T (L(w^*g w')) \hat p \geq 0$. 
When $g=1$, the matrix $(L(w^*w'))$ is indexed by the words of degree at most $r$, it is called the \emph{moment matrix} of $L$ and denoted  by $M_r(L)$ (or $M(L)$ when $r = \infty$). The entries of the matrices  $(L(w^*g w'))$ are linear combinations of the entries of $M_r(L)$, and the constraint $L=0$ on $\MI_{2r}(\mathcal H \cup \mathcal R_r)$ can be written as a set of linear constraints on the entries of $M_r(L)$. It follows that for finite $r \in \N$, the parameter $\xib{q}{r}(P)$ is indeed computable by a semidefinite program. 

\subsection{Background on positive tracial linear forms}
 
Before we show the convergence results we give some background on positive tracial linear forms, which we use again in Section~\ref{sec:estquantumchrom}. We state these results using the variables $x_1,\ldots,x_n$, where we use the notation $\langle \bx \rangle = \langle x_1,\ldots,x_n\rangle$. The results stated below do not always appear in this way in the sources cited; we follow the presentation of~\cite{GdLL17a}, where full proofs for these results are also provided.

First we need a few more definitions. A polynomial $p \in \R\ncx$ is called symmetric if $p^*=p$, and we denote the set of symmetric polynomials by $\mathrm{Sym}\, \R\langle \bx \rangle$. Given $\mathcal G \subseteq \mathrm{Sym} \, \R\langle \bx \rangle$ and $\mathcal H \subseteq \R\ncx$, the set $\mathcal M(\mathcal G) + \mathcal I(\mathcal H)$ is called {\em Archimedean} if it contains the polynomial $R-\sum_{i=1}^n x_i^2$ for some $R > 0$. We will use the concept of a $C^*$-algebra, which for our purposes can be defined as a norm closed $*$-subalgebra of the space $\MB(H)$ of bounded operators on a complex Hilbert space $H$. We say that $\MA$ is \emph{unital} if it contains the identity operator (denoted $1$). An element $a \in \MA$ is called \emph{positive} if $a = b^*b$ for some $b \in \MA$. A linear form $\tau$ on a unital $C^*$-algebra $\MA$ is said to be a \emph{state} if $\tau(1)=1$ and $\tau$ is positive; that is, $\tau(a) \geq 0$ for all positive elements $a \in \MA$. We say that a state $\tau$ is tracial if $\tau(ab)=\tau(ba)$ for all $a,b\in\MA$. See, for example,~\cite{Blackadar06} for more information on $C^*$-algebras. 

The first result relates positive tracial linear forms to $C^*$-algebras; see~\cite{NPA12} for the noncommutative (eigenvalue) setting and~\cite{BKP16} for the tracial setting. 
\begin{theorem}
\label{propLinfinitedim}
Let $\mathcal G \subseteq \mathrm{Sym}\,\R\ncx$ and $\mathcal H \subseteq \R\ncx$ and assume that $\MM(\mathcal G)+ \mathcal I(\mathcal H)$ is Archimedean. For a linear form $L\in \R\ncx^*$, the following are equivalent:
\begin{itemize}
\item[(1)] $L$ is symmetric, tracial, nonnegative on $\MM(\mathcal G)$, zero on $\mathcal I(\mathcal H)$, and $L(1) = 1$;
\item[(2)] there is a unital $C^*$-algebra $\mathcal A$ with tracial state $\tau$ and ${\bf X} \in \mathcal A^n$ such that $g(\bX)$ is positive in $\mathcal A$ for all $g \in \mathcal G$, and $h(\bX) = 0$ for all $h \in \mathcal H$, with
\begin{equation}\label{eqLA}
L(p)=\tau(p({\bf X})) \quad \text{for all} \quad p\in \R\ncx.
\end{equation}
\end{itemize}
\end{theorem}
The following can be seen as the finite dimensional analogue of the above result. 
The proof of the unconstrained case ($\mathcal G = \mathcal H = \emptyset$) can be found in~\cite{BK12}, and for the constrained case in~\cite{BKP16}. Given a linear form $L \in \R\ncx^*$, recall that the moment matrix $M(L)$ is given by $M(L)_{u,v} = L(u^*v)$ for $u,v \in \ncx$. 
\begin{theorem}
\label{propLfinitedim}
Let $\mathcal G \subseteq \mathrm{Sym}\,\R\ncx$ and $\mathcal H \subseteq \R\ncx$. For $L\in \R\ncx^*$, the following are equivalent:
\begin{enumerate}
\item[(1)] $L$ is a symmetric, tracial, linear form with $L(1) =1$ that is nonnegative on $\MM(\mathcal G)$, zero on $\mathcal I(\mathcal H)$, and has $\mathrm{rank}(M(L)) < \infty$;
\item[(2)] there is a finite dimensional $C^*$-algebra $\mathcal A$ with a tracial state $\tau$ and ${\bf X} \in \mathcal A^n$ satisfying~\eqref{eqLA}, with $g(\bX)$ positive in $\mathcal A$ for all $g \in \mathcal G$ and $h(\bX) = 0$ for all $h \in \mathcal H$;
\item[(3)] $L$ is a convex combination of normalized trace evaluations at tuples ${\bf X}$ 
of Hermitian matrices that satisfy $g(\bX) \succeq 0$ for all $g \in \mathcal G$ and $h(\bX) = 0$ for all $h \in \mathcal H$.
\end{enumerate}
\end{theorem}

A truncated 
linear functional $L \in \R\ncx_{2r}$ is called \emph{$\delta$-flat} if the principal submatrix $M_{r-\delta}(L)$ of $M_r(L)$ indexed by monomials up to degree $r-\delta$ has the same rank as $M_r(L)$; $L$ is \emph{flat} if it is $\delta$-flat for some $\delta\geq 1$. The following result claims that any {\em flat} linear functional on a truncated polynomial space  can be extended to a linear functional $L$ on the full algebra of polynomials. It is due to Curto and Fialkow~\cite{CF96} in the commutative case and extensions to the noncommutative case can be found in~\cite{NPA10} (for eigenvalue optimization) and~\cite{BK12,KP16} (for trace optimization). 
\begin{theorem}\label{propextension}
Let $1 \leq \delta \leq t < \infty$, $\mathcal G \subseteq \mathrm{Sym}\,\R\ncx_{2r}$, and $\mathcal H \subseteq \R\ncx_{2r}$. If $L\in \R\ncx_{2r}^*$ is symmetric, tracial, $\delta$-flat, nonnegative on $\MM_{2r}(\mathcal G)$, and zero on $\mathcal I_{2r}(\mathcal H)$, then $L$ extends to a symmetric, tracial, linear form on $\R\ncx$ that is nonnegative on $\mathcal M(\mathcal G)$, zero on $\mathcal I(\mathcal H)$, and whose moment matrix has finite rank.
\end{theorem}

The following technical lemma, based on the Banach-Alaoglu theorem, is a well-known tool to show asymptotic convergence results in polynomial optimization.

\begin{lemma}\label{lemma:upperboundLw_new}
Let $\mathcal G \subseteq \mathrm{Sym}\, \R\ncx$, $\mathcal H \subseteq \R\ncx$, 
and assume that for some $d \in \N$ and $R>0$ we have $R-(x_1^2 + \cdots + x_n^2) \in \MM_{2d}(\mathcal G)+\mathcal I_{2d}(\mathcal H)$. 
 For $r\in \N$ assume $L_r \in \smash{\R\langle \bx \rangle_{2r}^*}$ is tracial, nonnegative on $\mathcal M_{2r}(\mathcal G)$ and zero on $\MI_{2r}(\mathcal H)$. 
Then $\smash{|L_r(w)|\le R^{|w|/2} L_r(1)}$ for all $w\in \ncx_{2r-2d+2}$. In addition, if $\sup_r \, L_r(1) < \infty$, then  $\smash{\{L_r\}}_r$ has a pointwise converging subsequence in $\smash{\R\langle \bx \rangle^*}$.
\end{lemma}

\subsection{Convergence results}

We first show equality $\xib{q}{*}(P) = A_{qc}(P)$, and then we consider convergence properties of the bounds $\xib{q}{r}(P)$ to the parameters  $\xib{q}{\infty}(P)$ and  $\xib{q}{*}(P)$.

\propstarisaq
\begin{proof}
We already know that $\xib{q}{*}(P) \leq A_{qc}(P)$. To show $\xib{q}{*}(P) \geq A_{qc}(P)$ we  let $L$ be feasible for $\xib{q}{*}(P)$, so that $L \geq 0 $ on $\mathcal M(\mathcal G)$ and $L = 0$ on $\mathcal I(\mathcal H \cup \mathcal R_\infty)$. By Theorem~\ref{propLfinitedim},
there exist finitely many scalars $\lambda_i \geq 0$, Hermitian matrix tuples $\mathbf{X}(i) = (X^a_s(i))_{a,s}$ and  $\mathbf{Y}(i)= (Y^b_t(i))_{b,t}$, and Hermitian matrices $Z_i$, so that $g(\mathbf{X}(i), \mathbf{Y}(i), Z_i) \succeq 0$ for all $g \in \mathcal G$, $h(\mathbf{X}(i), \mathbf{Y}(i), Z_i) = 0$ for all $h \in \mathcal H \cup \mathcal R_\infty$, and
\begin{equation}\label{eq:LWI}
L(p) = \smash{\sum_i} \lambda_i\, \mathrm{Tr}(p(\mathbf{X}(i), \mathbf{Y}(i), Z_i)) \quad \text{for all} \quad p \in \R\langle \bx, \by,  z\rangle.
\end{equation}
By Artin--Wedderburn theory we know that for each $i$ there is a unitary matrix $U_{i}$ such that $U_{i}\C\langle \mathbf X(i), \mathbf  Y(i), Z_i\rangle U^*_{i} = \bigoplus_k \C^{d_k \times d_k} \otimes I_{m_k}$. 
Hence, after applying this further block diagonalization we may assume that in the decomposition (\ref{eq:LWI}), for each $i$,
  $\C\langle \mathbf X(i), \mathbf  Y(i), Z_i\rangle$ is a full matrix algebra $\C^{d_{ i }\times d_{ i}}$.

Since $h(\mathbf{E}(i), \mathbf{F}(i), Z_i) = 0$ for all $h \in R_\infty \cup \{z-z^2\}$, the commutator $\big[Z_i u Z_i, Z_i v Z_i\big]$ vanishes for all $u, v \in \langle \mathbf{E}(i), \mathbf{F}(i), Z_i\rangle$ and hence for all $u,v \in \C\langle \mathbf{E}(i), \mathbf{F}(i), Z_i\rangle$. This means that 
$
[Z_i T_1 Z_i, Z_i T_2 Z_i] = 0$ for all $T_1,T_2 \in \C^{d_i \times d_i}.
$
Since $Z_i$ is a projector, there exists a unitary matrix $U_i$ such that 
$
U_iZ_iU_i^* = \mathrm{Diag}(1,\ldots,1,0,\ldots,0).
$
The above then implies that for all $T_1$ and $T_2$, the leading principal submatrices of size $\mathrm{rank}(Z_i)$ of $U_iT_1U_i^*$ and $U_iT_2U_i^*$ commute. This implies $\mathrm{rank}(Z_{i}) \leq 1$ and thus $\Tr(Z_i) \in \{0,1\}$. Let $I$ be the set of indices with $\mathrm{Tr}(Z_i) = 1$. Then $\sum_{i \in I} \lambda_i = \sum_i \lambda_i \, \mathrm{Tr}(Z_i) = L(z) = 1$.

For each $i\in I$ define $P_i=\smash{(\mathrm{Tr}(E_s^a(i) F_t^b(i) Z_i))}$, which is a quantum correlation in  $C_{qc}^{d_i}(\Gamma)$ because $\Tr(Z_i)=1$,  $\sum_a X^a_s(i)=\sum_b Y^b_t(i)=I$, and $[X^a_s(i),Y^b_t(i)]=0$ by the ideal conditions. 
We have $P=\sum_{i \in I} \lambda_i P_i$, so that $(P_i,\lambda_i)_{i\in I}$ forms a feasible solution to $A_{qc}(P)$ with objective value 
$
\sum_{i\in I} \lambda_iD_{qc}(P_i)\le  \sum_{i \in I} \lambda_i d_i \le \sum_{i} \lambda_i d_i = L(1).
$
\end{proof}

The problem $\xib{q}{r}(P)$ differs in two ways from a standard tracial optimization problem. It does not have the normalization  
$L(1) = 1$ (and instead minimizes $L(1)$), and it has ideal constraints $L = 0$ on $\mathcal I_{2r}(\mathcal R_r)$ where $\mathcal R_r$ depends on $r$. We show asymptotic convergence still holds.

\propconvergetinfty
\begin{proof}
First observe that $1-z^2$, $1-(x_s^a)^2$, $1-(y_t^b)^2 \in \MM_4(\mathcal G \cup \mathcal H_0)$, where $\mathcal H_0$ contains the symmetric polynomials in $\mathcal H$; i.e., omitting the commutators $[x^a_s, y^b_t]$.
Indeed, we have $1-z^2= (1-z)^2+2(z-z^2)$ and $1-(x_s^a)^2= (1-x_s^a)^2+2(1-x^a_s)x^a_s(1-x^a_s)+ 2x^a_s\big( \big(1-\sum_{a'} x^{a'}_s\big)+ \sum_{a'\ne a}x^{a'}_s\big)x^a_s$, and the same for $y_t^b$. Hence $R-z^2-\sum_{a,s}(x^a_s)^2-\sum_{b,t}(y^b_t)^2\in \mathcal M_4(\mathcal G\cup\mathcal H_0)$ for some $R>0$.
Fix $\epsilon>0$ and for each $r\in \N$ let $L_r$ be feasible for $\xib{q}{r}(P)$ with value $L_r(1)\le \xib{q}{r}(P)+\epsilon$.
As $L_r$ is tracial and zero on $\mathcal I_{2r}(\mathcal H_0)$, it follows (using the identity $p^* gp = pp^*g + [p^*g,p]$) that $L=0$ on $\mathcal M_{2r}(\mathcal H_0)$. Hence,  $L_r\ge 0$ on $\mathcal M_{2r}(\mathcal G\cup \mathcal H_0)$.
Since $\sup_rL_r(1)\le A_q(P)+\epsilon$, we can apply Lemma~\ref{lemma:upperboundLw_new} and conclude that $\{L_r\}_r$ has a converging subsequence; denote its  limit by $L_\epsilon\in\R\ncx^*$.
One can verify that $L_\epsilon$ is feasible for $\xib{q}{\infty}(P)$, and $\xib{q}{\infty}(P)\le L_\epsilon(1)\le \lim_{r\to\infty} \xib{q}{r}(P) +\epsilon\le \xib{q}{\infty}(P)+\epsilon.$
Letting $\epsilon\to 0$ we obtain that $\xib{q}{\infty}(P)=\lim_{r\to\infty}\xib{q}{r}(P)$.
\end{proof}

Next we show that if $\xib{q}{r}(P)$ admits a $\delta$-flat optimal solution with $\delta =\lceil r/3 \rceil+1$, then we have $\xib{q}{r}(P) = \xib{q}{*}(P)$. This result is a variation 
of the flat extension result 
from Theorem~\ref{propextension}, where 
 $\delta$ now depends on the order $r$ because the ideal constraints in $\xib{q}{r}(P)$ depend on $r$.
\propqflat

\begin{proof}
Let $\delta = \lceil r/3 \rceil+1$ and let $L$ be a $\delta$-flat optimal solution to $\xib{q}{r}(P)$. We have to show $\xib{q}{r}(P) \geq \xib{q}{*}(P)$, which we do by constructing a feasible solution to $\xib{q}{*}(P)$ with the same objective value. The main step in the proof of Theorem~\ref{propextension}
consists of extending  the linear form $L$  to a tracial symmetric linear form $\smash{\hat L}$ on $\R\langle \mathbf x, \mathbf y, z \rangle$ that is nonnegative on $\mathcal M_{}(\mathcal G)$, zero on $\mathcal I(\mathcal H)$, and satisfies $\mathrm{rank}(M(\hat L)) < \infty$ 
(see the proof of~\cite[Thm.~2.3]{GdLL17a} for a detailed exposition).
 To do this a subset $W$ of $\langle \mathbf x, \mathbf y, z \rangle_{r-\delta}$ is found such that we have the vector space direct sum
$
\R\langle \mathbf x, \mathbf y, z \rangle = \mathrm{span}(W) \oplus \mathcal I(N_r(L)),
$
where $N_r(L)$ is the vector space
\[
N_r(L) = \big\{ p \in \R\langle \mathbf x, \mathbf y, z \rangle_r : L(qp) = 0 \text{ for all } q \in \R\langle \mathbf x, \mathbf y, z \rangle_r \big\}.
\]
It is moreover shown that $\mathcal I(N_r(L)) \subseteq N(\hat L)$. For $p \in \R\langle \mathbf x, \mathbf y, z \rangle$ we denote by $r_p$ the unique element in $\mathrm{span}(W)$ such that $p - r_p \in \mathcal I(N_r(L))$.
We now show that $\hat L$ is zero on $\mathcal I(\mathcal R_\infty)$.
For this fix $u,v, w \in \R\langle \mathbf x, \mathbf y, z\rangle$. 
Then  we have
\[
\hat L(w(z u z v z - z v z u z)) = \hat L(w z uz v z )  -\hat L(wz v z u z ).
\]
Since $\hat L$ is tracial and $u -r_u, v - r_v, w-r_w \in \mathcal I(N_r(L)) \subseteq N(\hat L)$, we have
\[
\hat L(wz u z vz) = \hat L(r_w z r_u z r_v z) \quad \text{and} \quad \hat L(wz v z uz) = \hat L(r_w z r_v z r_u z).
\]
Since $\mathrm{deg}(r_u z r_v z r_wz) = \mathrm{deg}(r_v z r_u z r_wz) \leq 2r$ we have
\[
\hat L(r_w z r_u z r_v z) = L(r_w z r_u z r_v z) \quad \text{and} \quad \hat L(r_w z r_v z r_u z) = L(r_w z r_v z r_u z).
\]
So $L=0$ on $ \mathcal I_{2r}(\mathcal R_r)$ implies $\hat L=0$ on $ \mathcal I(\mathcal R_\infty)$.

Since $\hat L$ extends $L$ we  have $\hat L(z) = L(z) = 1$ and $\hat L(x_s^ay_t^bz) = L(x_s^ay_t^bz) = P(a,b|s,t)$ for all $a,b,s,t$. So, $\hat L$ is feasible for $\xib{q}{*}(P)$ and has the same objective value $\hat L(1) = L(1)$.
\end{proof}

\section{Bounding quantum graph parameters} 
\label{sec:estquantumchrom}

\subsection{Hierarchies \texorpdfstring{$\gamma_r^\mathrm{col}(G)$ and $\gamma_r^\mathrm{stab}(G)$}{gamma col and gamma stab} based on synchronous correlations}\label{sec:gamma}
In Section~\ref{intro: quantum graph parameters} we introduced quantum chromatic numbers (Definition~\ref{defchiq}) and  quantum stability numbers (Definition~\ref{defaq}) in terms of synchronous  quantum correlations satisfying certain linear constraints. 
We first give (known) reformulations in terms of $C^*$-algebras, and then we reformulate those in terms of tracial optimization, which leads to the hierarchies $\gamma_r^\mathrm{col}(G)$ and $\gamma_r^\mathrm{stab}(G)$.

The following result from~\cite{PSSTW16} allows us to write a synchronous quantum correlation in terms of $C^*$-algebras admitting a tracial state.

\begin{theorem}[\cite{PSSTW16}]\label{theoCqcs}
Let $\Gamma = A^2 \times S^2$ and  $P\in\R^\Gamma$. We have $P \in C_{qc,s}(\Gamma)$ (resp., $P \in C_{q,s}(\Gamma)$) if and only if there exists a unital (resp., finite dimensional) $C^*$-algebra $\MA$ with a faithful tracial state $\tau$ and a set of projectors $\{X_s^a: s \in S, a \in A\} \subseteq \MA$  satisfying $\sum_{a \in A} X_s^a = 1$ for all $s \in S$ and
$P(a,b|s,t) = \tau(X_s^a X_t^b)$ for all $s,t \in S, a,b \in A$.
\end{theorem}
Here we add the condition that $\tau$ is faithful, that is, $\tau(X^*X) = 0$ implies $X=0$, since it follows from the GNS construction in the proof of~\cite{PSSTW16}.
 This means that
\[
0 = P(a,b|s,t) = \tau(X_s^a X_t^b) = \tau((X_s^a)^2 (X_t^b)^2) = \tau( (X_s^a X_t^b)^* X_s^a X_t^b)
\]
implies $X_s^a X_t^b = 0$. 
It follows that $\chi_{qc}(G)$ is equal to the smallest $k \in \N$ for which there exists a $C^*$-algebra $\MA$, a tracial state $\tau$ on $\MA$, and a family of projectors $\{X_i^c: i \in V, c \in [k]\}\subseteq \MA$ satisfying 
\begin{equation}\label{eqc1}
\sum_{c \in [k]} X_i^c -1 = 0 \quad \text{for all} \quad i \in V, 
\end{equation}
\vspace{-0.7em}
\begin{equation}\label{eqc2}
X_i^c X_j^{c'} = 0 \quad \text{if} \quad (c \ne c'  \text{ and }i=j ) \quad \text{or} \quad (c=c' \text{ and } \{i,j\}\in E).
\end{equation}
The quantum chromatic number $\chi_q(G)$ is equal to the smallest $k \in \N$ for which there exists a finite dimensional $C^*$-algebra $\MA$ with the above properties. 

Analogously, $\alpha_{qc}(G)$ is equal to the largest $k \in \N$ for which there is a $C^*$-algebra $\MA$, a tracial state $\tau$ on $\MA$, and a set of projectors $\{X_c^i: c \in [k], i \in V\}\subseteq \MA$ satisfying 
\begin{equation}\label{eqa1}
\sum_{i \in V} X_c^i -1 = 0 \quad \text{for all} \quad c \in [k], 
\end{equation}
\vspace{-0.7em}
\begin{equation}\label{eqa2}  
X_c^i X_{c'}^j = 0 \quad \text{if } (i \ne j \text{ and } c=c') \quad \text{or} \quad ((i=j \text{ or } \{i,j\}\in E) \text{ and } c\ne c'),
\end{equation}
and $\alpha_q(G)$ is equal to the largest $k \in \N$ for which $\mathcal A$ can be taken finite dimensional. 

These reformulations of $\chi_q(G), \chi_{qc}(G), \alpha_q(G)$ and $\alpha_{qc}(G)$ also follow from~\cite[Thm.~4.7]{OP16}, where general quantum graph homomorphisms are considered; the formulations of $\chi_q(G)$ and $\chi_{qc}(G)$ are also made explicit in~\cite[Thm.~4.12]{OP16}.

By Artin-Wedderburn theory~\cite{Wed,BEK78}, a finite dimensional $C^*$-algebra is isomorphic to a matrix algebra. So the above reformulations of $\chi_q(G)$ and $\alpha_q(G)$ can be seen as feasibility problems of systems of equations in matrix variables of unspecified (but finite) dimension; such formulations are given in~\cite{CMNSW07,MR16,SV16}. 
Restricting to scalar solutions ($1 \times 1$ matrices) in these feasibility problems recovers the classical graph parameters $\chi(G)$ and $\alpha(G)$.

\bigskip

We now reinterpret the above formulations in terms of tracial optimization. Given a graph $G = (V,E)$, let $i \simeq j$ denote $\{i,j\} \in E$ or $i =j$.  For $k \in \N$, let $\mathcal H_{G,k}^{\rm col}$ and $\mathcal H_{G,k}^{\rm stab}$ denote  the sets of polynomials corresponding to equations~\eqref{eqc1}--\eqref{eqc2} and~\eqref{eqa1}--\eqref{eqa2}: 
\[
\mathcal H^{\rm col}_{G,k} =\big\{1-\sum_{c\in [k]}x^c_i : i\in V\big\} \cup \big\{ x^c_ix^{c'}_j :(c \ne c'  \text{ and }i=j ) \text{ or } (c=c' \text{ and } \{i,j\}\in E) \big\}, 
\]
\vspace{-0.2em}
\[
\mathcal H^{\rm stab}_{G,k} = \big\{ 1-\sum_{i\in V}x^i_c  : c\in [k]\big\} \cup \big\{x^i_cx^j_{c'} : (i \ne j \text{ and } c=c')  \text{ or } (i \simeq j \text{ and } c\ne c')\big\}.
\]
We have $1-(x_i^c)^2 \in \MM_2(\emptyset) + \MI_2(\mathcal H_{G,k}^{\rm col})$, since $1-(x_i^c)^2 = (1-x_i^c)^2 + 2 (x_i^c - (x_i^c)^2)$ and
$
x_i^c-(x_i^c)^2= x^c_i\big(1-\sum_{c'}x^{c'}_i\big)+\sum_{c' : c'\ne c}x^c_ix^{c'}_i \in \mathcal I_2(\mathcal H^{\rm col}_{G,k})$,
and the analogous statements hold for $\mathcal H^\mathrm{stab}_{G,k}$. 
Hence, both $\MM(\emptyset) + \MI(\mathcal H_k^{\rm col})$ and $\MM(\emptyset) + \MI(\mathcal H_k^{\rm stab})$ are Archimedean and we can apply Theorems~\ref{propLinfinitedim} and~\ref{propLfinitedim} to express the quantum graph parameters in terms of positive tracial linear functionals.
Namely, 
\begin{align*}
\chi_{qc}(G) = \min \big\{k \in \N : \; & L \in \R\langle \{x_i^c: i \in V, c \in [k]\} \rangle^* \text{ symmetric, tracial, positive,} \\
&L(1) = 1, \, L=0 \text{ on } \MI(\mathcal H_{G,k}^{\rm col}) \big\},
\end{align*}
and $\chi_q(G)$ is obtained by adding the constraint $\rank(M(L)) < \infty$. Likewise, 
\begin{align*}
\alpha_{qc}(G) = \min \big\{k \in \N : \; & L \in \R\langle \{x_c^i: c \in [k], i \in V\} \rangle^* \text{ symmetric, tracial, positive,} \\
&L(1) = 1, \, L=0 \text{ on } \MI(\mathcal H_{G,k}^{\rm stab}) \big\},
\end{align*}
and $\alpha_q(G)$  is given by this program with the additional constraint $\rank(M(L)) <\infty$. 

Starting from these formulations it is natural to define a hierarchy $\smash\{\gamma_r^\mathrm{col}(G)\}$ of lower bounds on $\chi_{qc}(G)$ and a hierarchy $\smash\{\gamma_r^\mathrm{stab}(G)\}$ of upper bounds on $\alpha_{qc}(G)$,  where the  bounds of order $r\in \N$ are obtained by truncating $L$ to polynomials of degree at most $2r$ and truncating the ideal to degree $2r$. Then, by defining $\gamma^{\rm col}_*(G)$ and $\gamma^{\rm stab}_*(G)$ by adding the constraint $\mathrm{rank}(M(L)) < \infty$ to $\gamma^{\rm col}_\infty(G)$ and $\gamma^{\rm stab}_\infty(G)$, we have
\[
\gamma^{\rm col}_\infty(G)=\chi_{qc}(G), \ \  \gamma^{\rm stab}_\infty(G)=\alpha_{qc}(G), \ \  \gamma^{\rm col}_*(G)=\chi_q(G), \ \ \text{and} \quad \gamma^{\rm stab}_*(G) = \alpha_q(G).
\]

The optimization problems $\gamma_r^{\rm col}(G)$, for $r \in \N$, 
can be computed by semidefinite programming and binary search on $k$, since the positivity condition on $L$ can be expressed by requiring that its truncated moment matrix $M_r(L)=(L(w^*w'))$ (indexed by words with degree at most $r$) is positive semidefinite. 
If there is an optimal solution $(k,L)$ to $\gamma^{\rm col}_r(G)$ with $L$ flat, then, by Theorem~\ref{propextension}, we have equality $\gamma_r^\mathrm{col}(G) = \chi_q(G)$. Since $\smash\{\gamma_r^\mathrm{col}(G)\}_{r\in \N}$ is a monotone nondecreasing sequence of lower bounds on $\chi_q(G)$, there exists an $r_0$ such that for all $r \geq r_0$ we have $\gamma_r^\mathrm{col}(G) = \gamma_{r_0}^\mathrm{col}(G)$, which is equal to $\gamma^{\rm col}_\infty(G) =  \chi_{qc}(G)$ by Lemma~\ref{lemma:upperboundLw_new}. The analogous statements hold for the parameters $\gamma_r^\mathrm{stab}(G)$. 
Hence, we have shown the following result.

\LemConvergenceQuantum

\begin{remark} 
A  hierarchy  $\{Q_r(\Gamma)\}$ of outer semidefinite approximations for the set $C_{qc}(\Gamma)$ of commuting quantum correlations was constructed in \cite{PSSTW16}, revisiting the approach in~\cite{NPA08,NPA10}. This hierarchy is converging, that is,
\[
C_{qc}(\Gamma)=\mathcal Q_{\infty}(\Gamma)=\bigcap_{r\in \N}\mathcal Q_{r}(\Gamma).
\]
The approximations  $Q_r(\Gamma)$  are based on the eigenvalue optimization approach, applied to the formulation (\ref{eqPcommute}) of commuting quantum correlations, and thus they use linear functionals  on   polynomials  involving two sets of variables $x_s^a,y_t^b$ for $(a,b,s,t)\in\Gamma$. 
The authors of~\cite{PSSTW16} use these outer approximations of $C_{qc}(\Gamma)$  to define a converging hierarchy of lower bounds on $\chi_{qc}(G)$ in terms of feasibility  problems over the sets $Q_r(\Gamma)$.

 For synchronous correlations we can use the result of  Theorem~\ref{theoCqcs} and the tracial optimization approach used here to  define directly a converging hierarchy  $\{Q_{r,s}(\Gamma)\}$ of outer semidefinite approximations for the set  $C_{qc,s}(\Gamma)$ of synchronous commuting quantum correlations.  These approximations now use linear functionals  on polynomials involving only one set of variables~$x_s^a$. Namely, define $\mathcal Q_{r,s}(\Gamma)$ as the set of $P\in \R^{\Gamma}$ for which there exists a symmetric, tracial, positive linear functional $L\in \R\langle \{x_s^a: (a,s)\in A\times S\}\rangle_{2r}^*$ 
such that $L(1)=1$ and $L=0$ on the  ideal generated by the polynomials 
$x_s^a-(x_s^a)^2$ $((a,s)\in A\times S)$ and $1-\sum_{a\in A} x_s^a$ $(s\in S)$, truncated at degree $2r$. Then we have 
\[
C_{qc,s}(\Gamma)=\mathcal Q_{\infty,s}(\Gamma)=\bigcap_{r\in \N}\mathcal Q_{r,s}(\Gamma).
\]
The synchronous value of a nonlocal game is defined in~\cite{DP16} as the maximum value of the objective function~\eqref{eqvalue} over the set $C_{qc,s}(\Gamma)$.
By maximizing the objective~\eqref{eqvalue} over the relaxations $\mathcal Q_{r,s}(\Gamma)$ we get a hierarchy of semidefinite programming upper bounds that converges to the synchronous value of the game.
Finally note that one can also view the parameters $\gamma_r^\mathrm{col}(G)$ as solving feasibility problems over the sets $Q_{r,s}(\Gamma)$.
\end{remark}


\subsection{Hierarchies \texorpdfstring{$\xi_r^\mathrm{col}(G)$ and $\xi_r^\mathrm{stab}(G)$}{xi col and xi stab} based on Lasserre type bounds} \label{sec:nc analogues of lasserre type bounds}

Here we revisit some known Lasserre type hierarchies for the classical stability number $\alpha(G)$ and  chromatic number $\chi(G)$ and we 
show that their tracial noncommutative analogues can be used to recover known parameters such as the projective packing number $\alpha_p(G)$, the projective rank $\xi_f(G)$, and the tracial rank $\xi_{\rm tr}(G)$. Compared to the hierarchies defined in the previous section, these Lasserre type hierarchies use less variables (they only use  variables indexed by the vertices of the graph $G$), but they also do not converge to the (commuting) quantum chromatic or stability number.

Given a graph $G=(V,E)$, define the set of polynomials 
\[
\mathcal H_G = \big\{ x_i - x_i^2: i \in V\big\} \cup \big\{ x_i x_j: \{i,j\} \in E\big\}
\]
in the variables $\bx=(x_i: i\in V)$ (which are commutative or noncommutative depending on the context).
Note that $1-x_i^2\in \MM_2(\emptyset) + \mathcal I_2(\mathcal H_G)$ for all $i\in V$, so 
$\MM(\emptyset) + \mathcal I(\mathcal H_G)$ is Archimedean.

\subsubsection{Semidefinite programming bounds on the projective packing number}

We first recall the Lasserre hierarchy of bounds for the classical stability number $\alpha(G)$. 
Starting from the formulation of  $\alpha(G)$  via the optimization problem
\[
\alpha(G) = \sup \Big\{ \sum_{i \in V} x_i : x\in \R^n, \ h(x) = 0 \text{ for } h \in \mathcal H_G\Big\},
\]
the $r$-th level of the  Lasserre hierarchy for $\alpha(G)$ (introduced in~\cite{Las01,Lau03}) is defined by 
\[
\las{stab}{r}(G)= \mathrm{sup} \Big\{L\big(\sum_{i \in V} x_i\big) : L\in \R\cx_{2r}^* \text{ positive}, \, L(1)=1,\, L= 0  \text{ on }  \MI_{2r}(\mathcal H_G)\Big\}.
\]
Then $\las{stab}{r+1}(G) \leq \las{stab}{r}(G)$, the first bound is Lov\'asz' theta number: $\las{stab}{1}(G)=\vartheta(G)$, and finite convergence to $\alpha(G)$ is shown in~\cite{Lau03}: $\las{stab}{\alpha(G)}(G) = \alpha(G).$ 

Roberson~\cite{Rob13} introduces  the \emph{projective packing number}
\begin{align}
\alpha_p(G) &= \sup \Big\{ \frac{1}{d}\sum_{i \in V} \rank X_i :  d \in \N,\, {\bf X} \in (\mathcal S^d)^n \text{ projectors}, \ X_i X_j = 0 \text{ for } \{i,j\} \in E \Big\} \notag \\
&= \mathrm{sup}\Big\{\frac{1}{d}\mathrm{Tr}\Big(\sum_{i \in V} X_i \Big) : d\in \N,\, {\bf X} \in (\mathcal S^d)^n, \, h({\bf X}) = 0 \text{ for } h\in \mathcal H_G \Big\} \label{eqap}
\end{align}
as an upper bound for the quantum stability number $\alpha_q(G)$; the inequality  $\alpha_q(G)\le\alpha_p(G)$ also follows from Proposition~\ref{lemalphap} below. In view of~\eqref{eqap}, the parameter $\alpha_p(G)$ can be seen as a  noncommutative analogue of $\alpha(G)$.

For $r \in \N \cup \{\infty\}$ we define the noncommutative  analogue of $\las{stab}{r}(G)$ 
by 
\begin{align*}
\xib{stab}{r}(G) = \mathrm{sup}\Big\{L\Big(\sum_{i \in V} x_i\Big) : \; & L\in \R\langle {\bf x}\rangle_{2r}^* \text{ tracial, symmetric, and positive},
\\[-1em]
& L(1)=1,\, L = 0 \text{ on } \MI_{2r}(\mathcal H_G)  \Big\},
\end{align*}
and $\xib{stab}{*}(G)$ by adding the  constraint $\rank(M(L)) < \infty$ to the definition of  $\xib{stab}{\infty}(G)$. 

In view of Theorems~\ref{propLinfinitedim} and~\ref{propLfinitedim}, both $\xib{stab}{\infty}(G)$ and $\xib{stab}{*}(G)$ can be reformulated in terms of $C^*$-algebras: $\xib{stab}{\infty}(G)$ (resp., $\xib{stab}{*}(G)$) is the largest value of $\tau(\sum_{i\in V}X_i)$, where $\MA$ is a (resp., finite-dimensional) $C^*$-algebra with tracial state $\tau$ and  $X_i \in \MA$  ($i \in [n]$) are projectors satisfying $X_i X_j = 0$ for all $\{i,j\} \in E$.
Moreover, as we now see, the  parameter $\xib{stab}{*}(G)$ coincides with the projective packing number and the parameters $\xib{stab}{*}(G)$ and $\xib{stab}{\infty}(G)$ upper bound the quantum stability numbers.

\lemalphap
\begin{proof}
By~\eqref{eqap}, $\alpha_p(G)$ is the largest value of $L(\sum_{i\in V}x_i)$ over linear functionals $L$ that are normalized trace evaluations at projectors $\bX \in (\S^d)^n$ (for some $d \in \N$) with $X_i X_j = 0$ for $\{i,j\} \in E$.
By convexity the optimum remains unchanged when considering a convex combination of such trace evaluations. In view of Theorem~\ref{propLfinitedim}(3), this optimum is precisely the parameter $\xib{stab}{*}(G)$. This shows equality $\alpha_p(G)=\xib{stab}{*}(G)$.

Consider a $C^*$-algebra $\mathcal A$ with tracial state $\tau$ and projectors $X^i_c\in \mathcal A$ ($i\in V,\ c\in [k]$) satisfying~\eqref{eqa1}-\eqref{eqa2}.
Then, setting  $X_i=\sum_{c\in [k]} X^i_c$ for $i\in V$, we obtain projectors  $X_i\in \mathcal A$ that satisfy $X_iX_j=0$ if $\{i,j\}\in E$.
Moreover, $\tau(\sum_{i\in V}X_i)=\sum_{c\in [k]} \tau(\sum_{i\in V}X^i_c)=k$.
This shows $\xib{stab}{\infty}(G)\ge \alpha_{qc}(G)$ and, when restricting  $\mathcal A$ to be finite dimensional, $\xib{stab}{*}(G)\ge \alpha_q(G)$.
\end{proof}

Using Lemma~\ref{lemma:upperboundLw_new} one can verify  that $\xib{stab}{r}(G)$ converges to $\xib{stab}{\infty}(G)$ as $r~\rightarrow~\infty$, and  for $r \in \N \cup \{\infty\}$ the infimum in $\xib{stab}{r}(G)$ is attained.  Moreover, by Theorem~\ref{propextension}, if $\xib{stab}{r}(G)$ admits a flat optimal solution, then $\xib{stab}{r} = \xib{stab}{*}(G)$. 
Also, the first bound $\xib{stab}{1}(G)$ coincides with the theta number, since  $\xib{stab}{1}(G)=\las{stab}{1}(G)=\vartheta(G)$. Summarizing  we have 
$\alpha_{qc}(G)\le \xib{stab}{\infty}(G)$ and the following chain of inequalities
\[
\alpha_q(G)\leq \alpha_p(G)=\xib{stab}{*}(G)\leq \xib{stab}{\infty}(G)\leq \xib{stab}{r}(G)\leq  \xib{stab}{1}(G)=\vartheta(G).
\]

\subsubsection{Semidefinite programming bounds on the projective rank and tracial rank}

We now turn to the (quantum) chromatic numbers.
First recall the definition of the fractional chromatic number:
\[
\chi_f(G) := \min \Big\{ \sum_{S \in \mathcal S} \lambda_S : \lambda \in \R_+^\mathcal{S},\, \sum_{S\in \mathcal S: i\in S} \lambda_S = 1 \text{ for all } i\in V\Big\},
\]
where $\mathcal S$ is the set of stable sets of $G$. Clearly, $\chi_f(G)\le \chi(G)$. 
The following Lasserre type lower bounds for the classical chromatic number $\chi(G)$ are defined in~\cite{GL08}:
\[
\las{col}{r}(G) = \mathrm{inf} \big\{ L(1) : L\in \R\cx_{2r}^* \text{ positive},\,  L(x_i)=1\ (i \in V),\, L =0  \text{ on }  \mathcal I_{2r}(\mathcal H_G)\big\}. 
\]
By viewing $\chi_f(G)$ as minimizing $L(1)$ over linear functionals $L\in \R\cx^*$ that are conic combinations of evaluations at characteristic vectors of stable sets, we see that $ \las{col}{r}(G)\le \chi_f(G) $ for all $r \geq 1$. In~\cite{GL08} it is shown that $\las{col}{\alpha(G)}(G) = \chi_f(G)$. Moreover,  the order 1 bound coincides with the theta number:  $\las{col}{1}(G)=\vartheta(\overline G)$.

The  following parameter $\xi_f(G)$, called the \emph{projective rank}  of $G$, was introduced in~\cite{MR16} as a lower bound on the quantum chromatic number $\chi_q(G)$:
\begin{align*}
\xi_f(G) := \mathrm{inf} \big\{ \frac{d}{r} : \; & d,r\in \N,\ X_1, \ldots, X_n \in \mathcal S^d, \  \Tr(X_i)=r\ (i\in V), \\
& X_i^2 = X_i \ (i \in V), \ X_i X_j = 0\ (\{i,j\} \in E) 
 \big\}.
\end{align*}

\begin{proposition}[\cite{MR16}]\label{propprojrank}
For any graph $G$ we have $\xi_f(G)\le\chi_q(G)$.
\end{proposition}

\begin{proof}
Set $k=\chi_q(G)$. It is shown in~\cite{CMNSW07} that in the definition of $\chi_q(G)$ from~\eqref{eqc1}--\eqref{eqc2}, one may assume w.l.o.g.\ that all matrices $X^c_i$ have the same rank, say, $r$. Then, for any given color $c\in [k]$, the matrices $X^c_i$ ($i\in V$) provide a feasible solution to $\xi_f(G)$ with value $d/r$. Finally, $d/r=k$ holds since by~\eqref{eqc1}--\eqref{eqc2} we have $d=\rank (I)=\sum_{c=1}^k\rank (X^c_i)= kr$. 
\end{proof}

In~\cite[Prop. 5.11]{PSSTW16} it is shown that the projective rank can equivalently be defined as 
\begin{align*}
\xi_f(G) = \mathrm{inf} \big\{ \lambda : \; &\mathcal A \text{ is a finite dimensional } C^*\text{-algebra with tracial state } \tau,\\
&X_i \in \mathcal A \text{ projector with } \tau(X_i) = 1/\lambda \, (i\in V), \, X_i X_j = 0 \ (\{i,j\} \in E)\big\}. 
\end{align*}
They also define the \emph{tracial rank} $\xi_{tr}(G)$ of $G$ as the parameter obtained by omitting in the above definition of $\xi_f(G)$ the restriction that $\mathcal A$ has to be finite dimensional. The motivation for the parameter $\xi_{tr}(G)$  is that it lower bounds the \emph{commuting} quantum chromatic number~\cite[Thm.~5.11]{PSSTW16}: $\xi_{tr}(G)\le \chi_{qc}(G)$.

In view of Theorems~\ref{propLinfinitedim} and~\ref{propLfinitedim}, we  obtain the following reformulations: 
\begin{align*}
\xi_{f}(G) = \mathrm{inf} \big\{ L(1) : \; & L\in \R\langle {\bf x}\rangle^* \text{ tracial, symmetric, positive}, \, \rank(M(L))<\infty,\\
&L(x_i)=1\ (i \in V),\, L = 0 \text{ on } \mathcal I(\mathcal H_G)\big\},
\end{align*}
and $\xi_{tr}(G)$ is obtained by the same program without the restriction $\rank(M(L)) < \infty$. 
In addition, using Theorem~\ref{propLfinitedim}(3), we see that in this formulation of $\xi_f(G)$ 
 we can equivalently optimize over all $L$ that are conic combinations of trace evaluations at projectors $X_i \in \S^d$ (for some $d \in \N$)  satisfying $X_i X_j = 0$ for all  $\{i,j\} \in E$. 
If we restrict the optimization to {\em scalar} evaluations ($d=1$) we obtain the fractional chromatic number. This shows that the projective rank can be seen as the noncommutative analogue of the fractional chromatic number, as was already observed in~\cite{MR16,PSSTW16}.

The above formulations of the parameters $\xi_{tr}(G)$ and $\xi_f(G)$ in terms of linear functionals also show that they fit within the following hierarchy $\smash{\{\xib{col}{r}(G)\}_{r\in \N\cup\{\infty\}}}$, defined as the noncommutative tracial analogue of the hierarchy  $\{\las{col}{r}(G)\}_{r}$: 
\begin{align*}
\xib{col}{r}(G) = \mathrm{inf}  \big\{ L(1) : \; & L\in \R\langle {\bf x}\rangle_{2r}^* \text{ tracial, symmetric, and positive},  \\
&L(x_i)=1\ (i \in V),\, L = 0 \text{ on } \MI_{2r}(\mathcal H_G)  \big\}.
\end{align*}
Again, $\xib{col}{*}(G)$ is the parameter obtained by adding the constraint $\rank (M(L)) <\infty$ to the program defining $\xib{col}{\infty}(G)$. By the above discussion the following holds.
\lemchif
Using Lemma~\ref{lemma:upperboundLw_new} one can verify  that the parameters $\xib{col}{r}(G)$ converge to $\xib{col}{\infty}(G)$. 
Moreover, by Theorem~\ref{propextension}, if $\xib{col}{r}(G)$ admits a flat optimal solution, then we have $\xib{col}{r} = \xib{col}{*}(G)$. 
Also,  the parameter $\xib{col}{1}(G)$ coincides with $\las{col}{1}(G)=\vartheta(\overline G)$. Summarizing we have $\xib{col}{\infty}(G)=\xi_{tr}(G)\le \chi_{qc}(G)$ and the following chain of inequalities
\[
\vartheta(\overline G)=\xib{col}{1}(G) \le \xib{col}{r}(G)\le \xib{col}{\infty}(G)=\xi_{tr}(G)\le \xib{col}{*}(G)=\xi_f(G) \le \chi_q(G).
\]

Observe that the bounds $\las{col}{r}(G)$ and $\xib{col}{r}(G)$ remain below the fractional chromatic number $\chi_f(G)$, since
$\xi_f(G)= \xib{col}{*}(G)\le \las{col}{*}(G)=\chi_f(G)$. Hence, these bounds are weak if $\chi_f(G)$ is close to $\vartheta(\overline G)$ and far from $\chi(G)$ or $\chi_q(G)$. In the classical setting this is the case, e.g., for the class of Kneser graphs $G=K(n,r)$, with vertex set the set of  all $r$-subsets of $[n]$ and having an edge between any two disjoint $r$-subsets. By results of Lov\'asz~\cite{Lo78,Lo79}, the fractional chromatic number is $n/r$, 
which is known to be equal to $\vartheta(\overline {K(n,r)})$, while the chromatic number is $n-2r+2$. In~\cite{GL08} this was used as a motivation to define a new hierarchy  of lower bounds $\{\Lambda_r(G)\}$ on the chromatic number that can go beyond the fractional chromatic number. In Section~\ref{sec:chromatic via stability} we recall this approach and show that its extension  to the tracial setting recovers the hierarchy $\{\gamma_r^\mathrm{col}(G)\}$ introduced in Section~\ref{sec:gamma}. We also show how a similar technique can be used to recover the hierarchy $\{\gamma_r^\mathrm{stab}(G)\}$.

\subsubsection{A link between \texorpdfstring{$\xib{stab}{r}(G)$ and $\xib{col}{r}(G)$}{xi stab and xi col}} 
\label{seclinkG}

In~\cite[Thm.~3.1]{GL08} it is shown that, for any $r\ge 1$,  the bounds $\las{stab}{r}(G)$ and $\las{col}{r}(G)$ satisfy 
$\las{stab}{r}(G) \las{col}{r}(G) \geq |V|$, with equality if $G$ is vertex-transitive, which  extends a well-known property of the theta number (case $r=1$).  The same holds for the noncommutative analogues $\xib{stab}{r}(G)$ and $\xib{col}{r}(G)$. 

\begin{lemma}\label{lemtransitive}
For any graph $G=(V,E)$ and $r\in \N\cup\{\infty,*\}$ we have
$\xib{stab}{r}(G)\xib{col}{r}(G)\ge |V|,$
with equality if $G$ is vertex-transitive.
\end{lemma}

\begin{proof}
Let $L$ be feasible for $\xib{col}{r}(G)$. Then $\tilde L = L/L(1)$ provides a solution to $\xib{stab}{r}(G)$ with value 
$\tilde L\big(\sum_{i\in V}x_i\big)= |V|/L(1)$, implying that $\xib{stab}{r}(G)\ge |V|/L(1)$ and therefore $\xib{stab}{r}(G)\xib{col}{r}(G)\ge |V|$.

Assume  $G$ is vertex-transitive. Let $L$ be a feasible solution for $\xib{stab}{r}(G)$. As $G$ is vertex-transitive we may assume (after symmetrization) that  $L(x_i)$ is constant, set  $L(x_i)=:1/\lambda$ for all $i\in V$, so that the objective value of $L$ for $\xib{stab}{r}(G)$ is $|V|/\lambda$.  Then $\tilde L = \lambda L$ provides a  feasible solution for $\xib{col}{r}(G)$ with value $\lambda$, implying $\xib{col}{r}(G) \le \lambda$. This implies $\xib{col}{r}(G)\xib{stab}{r}(G)\le |V|$.
\end{proof}
For vertex-transitive $G$, the inequality $\xi_f(G)\alpha_q(G) \leq |V|$ is shown in \cite[Lem.~6.5]{MR16};  it can be recovered from the $r=*$ case of Lemma~\ref{lemtransitive}  and $\alpha_q(G) \leq \alpha_p(G)$. 

\subsubsection{Comparison to existing semidefinite programming bounds}\label{seccompare}

By adding the inequalities $L(x_i x_j) \geq 0$, for all $i,j \in V$, to $\xib{col}{1}(G)$, we obtain the strengthened theta number $\vartheta^+(\overline G)$ (from~\cite{Szegedy}). Moreover, if we add the constraints 
\begin{align}
L(x_ix_j) &\ge 0 & \text{ for } i\ne j\in V, \label{eqij}\\
\sum_{j\in C} L(x_ix_j) &\le 1 & \text{ for } i\in V,\label{eqiC}\\
L(1) + \sum_{i\in C, j\in C'}L(x_ix_j)&\geq |C| + |C'| & \text{ for } C,C' \text{ distinct cliques in } G\label{eqCC'}
\end{align}
to the program defining the parameter $\xib{col}{1}(G)$, then we obtain the parameter $ \xi_\mathrm{SDP}(G)$, which is  introduced  in~\cite[Thm.~7.3]{PSSTW16} as a lower bound on $\xi_\mathrm{tr}(G)$. We will now show that the inequalities~\eqref{eqij}--\eqref{eqCC'} are in fact valid for $\xib{col}{2}(G)$, which implies 
\[
\xib{col}{2}(G) \geq \xi_\mathrm{SDP}(G) \geq \vartheta^+(\overline G).
\]
For this, given a clique $C$ in $G$, we define the polynomial
$g_C:=1-\sum_{i\in C}x_i\in \R\ncx$.
Then~\eqref{eqiC} and~\eqref{eqCC'} can be reformulated as $L(x_i g_C) \geq 0$ and $L(g_C g_{C'}) \geq 0$, respectively, using the fact that $L(x_i) = L(x_i^2)=1$ for all $i \in V$. Hence, to show that any feasible $L$  for $\xib{col}{2}(G)$ satisfies~\eqref{eqij}-\eqref{eqCC'}, it suffices to show Lemma~\ref{lemineqSG} below.
Recall that a commutator is a polynomial of the form $[p,q]=pq-qp$ with $p,q\in\R\ncx$. We denote the set of linear combinations of commutators $[p,q]$ with $\deg(pq)\le r$ by $\Theta_{r}$. 

\begin{lemma}\label{lemineqSG}
Let  $C$ and $C'$ be  cliques in a graph $G$ and let $i,j\in V$. Then we have
\[
g_C \in \MM_2(\emptyset)+\MI_2(\mathcal H_G), \text{ and } \ x_ix_j,\  x_ig_C,\ g_Cg_{C'}  \in \MM_4(\emptyset)+\mathcal I_4(\mathcal H_G)+\Theta_4.
\]
\end{lemma}
\begin{proof}
The claim $g_C\in\MM_2(\emptyset)+\MI_2(\mathcal H_G)$ follows from the identity
\begin{equation} \label{eq:gC is a projector mod ideal}
g_C=\Big(\underbrace{ 1-\sum_{i\in C}x_i}_{g_C}\Big)^2  +\underbrace{\sum_{i\in C} (x_i-x_i^2) +\sum_{i\ne j\in C} x_ix_j}_{h}=g_C^2+h,
\end{equation}
where $h\in \mathcal I_2(\mathcal H_G)$.
We also have
\begin{align*}
x_ix_j &= x_ix_j^2x_i +x_j(x_i-x_i^2)+x_i^2(x_j-x_j^2) +[x_i,x_ix_j^2]+[x_i-x_i^2,x_j], \\
x_ig_C &= x_ig_C^2x_i+ g_C^2(x_i-x_i^2) + [x_i-x_i^2,g_C^2] +[x_i,x_ig_C^2], 
\end{align*}
and, writing analogously $g_{C'} =g_{C'}^2+h'$ with $h'\in \mathcal I_2(\mathcal H_G)$, we have
\[
g_Cg_{C'}= g_C g_{C'}^2 g_C + [g_C,g_C g_{C'}^2]+[h,g_{C'}^2]+g_C^2 h'+hh'+ g_{C'}^2 h.  \qedhere
\]
\end{proof}

Using $\xi_{\rm{SDP}}(G)$,  it is shown in~\cite[Thm.~7.4]{PSSTW16} that for the odd cycle $C_{2n+1}$, the tracial rank satisfies $\xib{col}{\infty}(C_{2n+1})=(2n+1)/n$. Combining this with Lemma~\ref{lemtransitive} gives $n = \xib{stab}{\infty}(C_{2n+1}) \geq \alpha_{qc}(C_{2n+1})$. Equality holds since $\alpha_{qc}(C_{2n+1})\ge \alpha(C_{2n+1})=n$. 

\subsection{Links between the bounds \texorpdfstring{$\gamma^{\rm col}_r(G)$, $\xi_r^\mathrm{col}(G)$, $\gamma^{\rm stab}_r(G)$, and $\xi_r^\mathrm{stab}(G)$}{gamma col, xi col, gamma stab, and xi stab}} \label{sec:chromatic via stability}

Here, in this last section, we make the link between the hierarchies $\{\xib{stab}{r}(G)\}$ (resp. $\{\xib{col}{r}(G)\}$) and $\{\gamma^{\rm stab}_r(G)\}$ (resp. $\{\gamma^{\rm col}_r(G)\}$). The key fact is the interpretation of the coloring and stability numbers in terms of certain graph products. 

We start with the (quantum) coloring number. For an integer $k$, recall that the Cartesian product $G\Box K_k$  is the graph with vertex set $V\times [k]$, where the vertices $(i,c)$ and $(j,c')$ are adjacent if ($\{i,j\}\in E$ and $c=c'$) or ($i=j$ and $c\ne c'$). 
The  following is a well-known reduction of the chromatic number $\chi(G)$  to the stability number of the Cartesian product $G\Box K_k$:
$
\chi(G)=\min\big\{k\in \N: \alpha(G \square K_k)=|V|\big\}$.
It was used  in~\cite{GL08} to define the following lower bounds on the chromatic number:
\[
\Lambda_r(G) = \min \big\{ k\in \N: \las{stab}{r}( G\Box K_k) = |V|\big\},
\]
where it was also shown that $\las{col}{r}(G) \leq \Lambda_r(G)\le \chi(G)$  for all $r\ge 1$, with  equality
$\Lambda_{|V|}(G)=\chi(G)$.
Hence the bounds  $\Lambda_r(G)$ may go beyond the fractional chromatic number. This is the case for the above mentioned Kneser graphs; see~\cite{GL08b} for other graph instances.

The above reduction from coloring to stability number has been extended to the quantum setting by~\cite{MR16}, where it is shown that
$\chi_q(G)=\min\{k\in \N: \alpha_q(G\Box K_k)=|V|\}$.
It is therefore natural to use the upper bounds  $\xib{stab}{r}(G\Box K_k)$ on $\alpha_q(G\Box K_k)$ in order to get  the following lower bounds on the quantum coloring number:
\begin{equation}\label{eqXicol}
\min\{k: \xib{stab}{r}(G\Box K_k)=|V|\},
\end{equation}
which are thus  the noncommutative analogues of the bounds $\Lambda_r(G)$.
Observe that,  for any integer $k\in \N$ and $r \in \N \cup \{\infty, *\}$, we have
$
\xib{stab}{r}(G\Box K_k)\le |V|,
$
which follows from Lemma~\ref{lemineqSG} and the fact that  the cliques $C_i=\{(i,c): c\in [k]\}$, for $i\in V$, cover all vertices in $G\Box K_k$. Let
$
\mathcal C_{G\Box K_k} = \big\{g_{C_i} : i\in V\big\}$, where $g_{C_i} = 1-\sum_{c\in [k]}x^c_i$,
denote the set of polynomials corresponding to  these cliques. We now show  that the parameters~\eqref{eqXicol} coincide in fact with $\gamma_r^\mathrm{col}(G)$ for all $r \in \N \cup \{\infty\}$. For this observe first that the quadratic polynomials in the set $\smash{\mathcal H^{\rm col}_{G,k}}$ correspond precisely to the edges of $G\Box K_k$, and the projector constraints are included in $\MI_{2}(\mathcal H_{G,k}^\mathrm{col})$ (see Section~\ref{sec:gamma}), so that 
$\smash{\MI_{2r}(\mathcal H^{\rm col}_{G,k}) = \mathcal I_{2r}(\mathcal H_{G\Box K_k} \cup\mathcal C_{G\Box K_k})}$.
We will also use the following result.
 
\begin{lemma}\label{lemXi}
Let $r\in \N\cup \{\infty,*\}$ and assume $L$ is feasible for $\xib{stab}{r}(G\Box K_k)$.
Then, we have $L(\sum_{i\in V, c\in [k]}x_i^c)=|V|$  if and only if $L=0$ on $\MI_{2r}(\mathcal C_{G\Box K_k})$.
\end{lemma}

\begin{proof}
First: If $L=0$ on $\MI_{2r}(\mathcal C_{G\Box K_k})$, then $0=\sum_{i\in V}L(g_{C_i})= |V|- L(\sum_{i,c}x_i^c)$.

Conversely assume that 
$
0= L\big(\sum_{i\in V, c\in [k]}x_i^c\big) - |V|=\sum_{i\in V} L(g_{C_i})$.
We will show  $L=0$ on $\MI_{2r}(\mathcal C_{G\Box K_k})$.
For this we first observe that $g_{C_i}-(g_{C_i})^2\in \MI_2(\mathcal H_{G\Box K_k})$ by~\eqref{eq:gC is a projector mod ideal}. 
Hence $L(g_{C_i})=L(g_{C_i}^2)\ge 0$, which, combined with $\sum_i L(g_{C_i})=0$, implies $L(g_{C_i})=0$ for all $i\in V$.
Next we show $L(wg_{C_i})=0$ for all  words $w$ with degree at most $2r-1$, using induction on $\deg(w)$. The base case $w=1$ holds by the above. Assume now 
$w=uv$, where  $\deg(v)<\deg(u)\le r$. Using the positivity of $L$, the Cauchy-Schwarz inequality gives
$|L(uvg_{C_i})| \le {L(u^*u)}^{1/2}{L(v^*g_{C_i}^2v)}^{1/2}$.
Note that it suffices to show $L(v^*g_{C_i}v)=0$ since, using again~\eqref{eq:gC is a projector mod ideal}, this implies  $L(v^*g_{C_i}^2v)=0$ and thus  $L(uvg_{C_i}) =0$.
Using the tracial property of $L$ and the induction assumption, we see that $L(v^*g_{C_i}v)=L(vv^*g_{C_i})=0$ since 
$\deg(vv^*)<\deg(w)$.
\end{proof}

\propXicol
\begin{proof}
Let $L$ be a linear functional certifying  $\gamma^{\rm col}_r(G) \leq k$. Then $L$ is feasible for $\smash{\xib{stab}{r}(G\Box K_k)}$ and, as $L=0$ on $\mathcal I_{2r}(\mathcal C_{G\Box K_k})$, Lemma~\ref{lemXi} shows that $L(\sum_{i,c}x_i^c)=|V|$. This shows that $\xib{stab}{r}(G\Box K_k)=|V|$ and thus $\min\{k: \xib{stab}{r}(G\Box K_k) =|V|\}\le k$.

Conversely, assume $\xib{stab}{r}(G\Box K_k)=|V|$. Since the optimum is attained,  there exists a linear functional $L$ feasible for $\xib{stab}{r}(G\Box K_k)$ with $L(\sum_{i,c} x_i^c)=|V|$. Using Lemma~\ref{lemXi} we can conclude that $L$ is zero on $\smash{\mathcal I_{2r}(\mathcal C_{G\Box K_k})}$ and  thus also on $\smash{\mathcal I_{2r}(\mathcal H^{\rm col}_{G,k})}$. This shows $\smash{\gamma^{\rm col}_r(G)}\le k$.
\end{proof}
Note that the proof of 
Proposition~\ref{propXicol} also works in the commutative setting; this shows  that the sequence $\Lambda_r(G)$ corresponds to the usual Lasserre hierarchy for the feasibility problem defined by the equations~\eqref{eqc1}--\eqref{eqc2}, which is another way of showing $\Lambda_\infty(G) = \chi(G)$.  

We now turn to the (quantum) stability number. For $k \in \N$, consider the graph product $K_k\star G$, with vertex set $[k]\times G$, and with an edge between $(c,i)$ and $(c',j)$ when 
$(c\ne c',i=j)$ or $(c=c', i\ne j)$ or $(c\ne c', \{i,j\}\in E)$. The product $K_k\star G$ coincides with the homomorphic product $K_k\ltimes \overline G$ used in~\cite[Sec.~4.2]{MR16}, where it is shown that 
$
\alpha_q(G)=\max \big\{k \in \N: \alpha_q(K_k\star G)= k\big\}$.
This suggests using the upper bounds $\xib{stab}{r}(K_k\star G)$ on $\alpha_q(K_k\star G)$ to define the following upper bounds on $\alpha_q(G)$:
\begin{equation}\label{eqXistab}
\max\big\{k\in \N: \xib{stab}{r}(K_k\star G)=k\big\}.
\end{equation}
For each $c \in [k]$, the set $C^c=\{(c,i):i\in V\}$ is a clique in $K_k\star G$ and we let 
$\mathcal C_{K_k\star G}=\big\{g_{C^c} : c\in [k]\big\}$, where $g_{C^c} = 1-\sum_{i\in V}x^i_c$,
denote the set of polynomials corresponding to these cliques.
Since these $k$ cliques cover the vertex set of $K_k\star G$, we can use Lemma~\ref{lemineqSG}  to conclude
 $\xib{stab}{r}(K_k\star G)\le k$ for all $r\in\N\cup\{\infty,*\}$.
Again, observe that the quadratic polynomials in the set $\mathcal H^{\rm stab}_{G,k}$ correspond precisely to the edges of $K_k\star G$ and  that we have 
$
\mathcal I_{2r}(\mathcal H^{\rm stab}_{G,k})= \mathcal I_{2r}(\mathcal H_{K_k\star G}\cup \mathcal C_{K_k\star G}).
$
Based on this, one can show the analogue of Lemma~\ref{lemXi}: If $L$ is feasible for the program $\xib{stab}{r}(K_k\star G)$, 
then we have $L(\sum_{i,c}x^i_c)=k$ if and only if $L=0$ on $\mathcal I_{2r}(\mathcal C_{K_k\star G})$, which implies the following result.

\propXistab

We do not know whether the results of Propositions~\ref{propXicol} and~\ref{propXistab} hold for $r=*$, since we do not know whether the supremum is attained in the parameter $\xib{stab}{*}(\cdot)=\alpha_p(\cdot)$ (as was already observed in~\cite[p.~120]{Rob13}).
 Hence we can only claim the inequalities
\[
\gamma^{\rm col}_*(G)\ge 
\min\{k: \xib{stab}{*}(G\Box K_k)=|V|\} \quad \text{and} \quad 
\gamma^{\rm stab}_*(G) \leq 
\max\{k: \xib{stab}{*}(K_k\star G)=k\}.
\]

As mentioned above, we have $\las{col}{r}(G)\le \Lambda_r(G)$ for any $r\in\N$~\cite[Prop.~3.3]{GL08}. This result extends to the noncommutative setting and  the analogous result holds for the stability parameters. In other words the hierarchies $\{\gamma^{\rm col}_r(G)\}$ and $\{\gamma^{\rm stab}_r(G)\}$ refine the hierarchies $\{\xib{col}{r}(G)\}$ and $\xib{stab}{r}(G)\}$.

\propcolstabcompare
\begin{proof}
We may restrict to  $r\in \N$ since we have seen earlier that the inequalities hold for $r\in \{\infty,*\}$.
The proof for the coloring parameters is similar to the proof of~\cite[Prop.~3.3]{GL08} in the classical case and thus omitted. We show the inequality 
$ \xib{stab}{r}(G)\ge \gamma^{\rm stab}_r(G)$.
Set $k=\gamma^{\rm stab}_r(G)$ and, using Proposition~\ref{propXistab}, let $L\in \R\langle x^i_c: i\in V, c\in [k]\rangle_{2r}^*$ be optimal for
$\xib{stab}{r} (K_k\star G)=k$. That is, $L$ is tracial, symmetric, positive, and satisfies
$L(1)=1$, $L(\sum_{i,c}x^i_c)=k$,  and $L=0$ on 
$\mathcal I(\mathcal H_{K_k\star G}).$ It suffices now to  construct a tracial symmetric positive linear form $\hat L\in \R\langle x_i:i\in V\rangle_{2r}^*$  such that $\hat L(1)=1$, $\hat L(\sum_{i\in V}x_i)=k$, and $\hat L = 0$ on $\mathcal I_{2r}(\mathcal H_G)$, since this will imply $\xib{stab}{r}(G)\ge k$. For this, for any word $x_{i_1}\cdots x_{i_t}$ with degree $1 \leq t \le 2r$, we define
$
\hat L(x_{i_1}\cdots x_{i_t}) :=  \sum_{c\in [k]} L(x^{i_1}_c \cdots x^{i_t}_c)$.
Also, we set $\hat L(1) = L(1) = 1$. Then, we have 
$\hat L(\sum_{i\in V}x_i)=k$.
Moreover, one can easily check that $\hat L$ is indeed  tracial, symmetric, positive, and  vanishes on $\mathcal I_{2r}(\mathcal H_G)$.
\end{proof}


\begin{thebibliography}{10}

\bibitem{AHKS06}
D.~Avis, J.~Hasegawa, Y.~Kikuchi, and Y.~Sasaki.
\newblock {A quantum protocol to win the graph coloring game on all Hadamard
  graphs}.
\newblock {\em IEICE Transactions on Fundamentals of Electronics,
  Communications and Computer Sciences}, E89-A(5):1378--1381, 2006.

\bibitem{BEK78}
G.~P. Barker, L.~Q. Eifler, and T.~P. Kezlan.
\newblock A non-commutative spectral theorem.
\newblock {\em Linear Algebra and its Applications}, 20(2):95--100, 1978.

\bibitem{B64}
J.~S. Bell.
\newblock {On the Einstein Podolsky Rosen paradox}.
\newblock {\em Physics}, 1(3):195--200, 1964.

\bibitem{Blackadar06}
B.~Blackadar.
\newblock {\em Operator Algebras: Theory of $C^*$-Algebras and Von Neumann
  Algebras}.
\newblock Encyclopaedia of Mathematical Sciences. Springer, 2006.

\bibitem{Brunner08}
N.~Brunner, S.~Pironio, A.~Acin, N.~Gisin, A.~A. M\'ethot, and V.~Scarani.
\newblock {Testing the dimension of Hilbert spaces}.
\newblock {\em Physical Review Letters}, 100:210503, 2008.

\bibitem{BCKP13}
S.~Burgdorf, K.~Cafuta, I.~Klep, and J.~Povh.
\newblock The tracial moment problem and trace-optimization of polynomials.
\newblock {\em Mathematical Programming}, 137(1):557--578, 2013.

\bibitem{BK12}
S.~Burgdorf and I.~Klep.
\newblock The truncated tracial moment problem.
\newblock {\em Journal of Operator Theory}, 68(1):141--163, 2012.

\bibitem{BKP16}
S.~Burgdorf, I.~Klep, and J.~Povh.
\newblock {\em {Optimization of Polynomials in Non-Commutative Variables}}.
\newblock Springer Briefs in Mathematics. Springer, 2016.

\bibitem{CMNSW07}
P.~J. Cameron, A.~Montanaro, M.~W. Newman, S.~Severini, and A.~Winter.
\newblock On the quantum chromatic number of a graph.
\newblock {\em The Electronic Journal of Combinatorics}, 14(1), 2007.

\bibitem{CF96}
R.~E. Curto and L.~A. Fialkow.
\newblock {\em Solution of the Truncated Complex Moment Problem for Flat Data},
  volume 568 of {\em Memoirs of the American Mathematical Society}.
\newblock American Mathematical Society, 1996.

\bibitem{DLTW08}
A.C. Doherty, Y.-C. Liang, B.~Toner, and S.~Wehner.
\newblock The quantum moment problem and bounds on entangled multiprover games.
\newblock {\em Proceedings of the 2008 IEEE 23rd Annual Conference on
  Computational Complexity}, pages 199--210, 2008.

\bibitem{DP16}
K.~J. Dykema and V.~I. Paulsen.
\newblock {Synchronous correlation matrices and Connes' embedding conjecture}.
\newblock {\em Journal of Mathematical Physics}, 57:015214, 2016.

\bibitem{DPP17}
K.~J. Dykema, V.~I. Paulsen, and J.~Prakash.
\newblock Non-closure of the set of quantum correlations via graphs.
\newblock {\em {arXiv:1709.05032}}, 2017.

\bibitem{Fri12}
T.~Fritz.
\newblock {Tsirelson's problem and Kirchberg's conjecture}.
\newblock {\em Reviews in Mathematical Physics}, 24(05), 2012.

\bibitem{GdLL17a}
S.~Gribling, D.~de~Laat, and M.~Laurent.
\newblock Lower bounds on matrix factorization ranks via noncommutative
  polynomial optimization.
\newblock {\em arXiv:1708.01573}, 2017.

\bibitem{GdLL17}
S.~Gribling, D.~de~Laat, and M.~Laurent.
\newblock Matrices with high completely positive semidefinite rank.
\newblock {\em Linear Algebra and its Applications}, 513:122--148, 2017.

\bibitem{GL08b}
N.~Gvozdenovi{\'c} and M.~Laurent.
\newblock Computing semidefinite programming lower bounds for the (fractional)
  chromatic number via block-diagonalization.
\newblock {\em SIAM Journal on Optimization}, 19(2):592--615, 2008.

\bibitem{GL08}
N.~Gvozdenovi{\'c} and M.~Laurent.
\newblock The operator $\psi$ for the chromatic number of a graph.
\newblock {\em SIAM Journal on Optimization}, 19(2):572--591, 2008.

\bibitem{Ji13}
Z.~Ji.
\newblock Binary constraint system games and locally commutative reductions.
\newblock {\em arXiv:1310:3794}, 2013.

\bibitem{JNPPGSW}
M.~Junge, M.~Navascues, C.~Palazuelos, D.~Perez-Garcia, V.B. Scholtz, and R.F.
  Werner.
\newblock {Connes' embedding problem and Tsirelson's problem}.
\newblock {\em Journal of Mathematical Physics}, 52:012102, 2011.

\bibitem{KP16}
I.~Klep and J.~Povh.
\newblock Constrained trace-optimization of polynomials in freely noncommuting
  variables.
\newblock {\em Journal of Global Optimization}, 64(2):325--348, 2016.

\bibitem{KS08}
I.~Klep and M.~Schweighofer.
\newblock {Connes' embedding conjecture and sums of Hermitian squares}.
\newblock {\em Advances in Mathematics}, 217(4):1816--1837, 2008.

\bibitem{Las01}
J.~B. Lasserre.
\newblock Global optimization with polynomials and the problem of moments.
\newblock {\em SIAM Journal on Optimization}, 11(3):796--817, 2001.

\bibitem{Lau03}
M.~Laurent.
\newblock A comparison of the {S}herali-{A}dams, {L}ov\'asz-{S}chrijver, and
  {L}asserre relaxations for 0-1 programming.
\newblock {\em Mathematics of Operations Research}, 28(3):470--496, 2003.

\bibitem{LP15}
M.~Laurent and T.~Piovesan.
\newblock Conic approach to quantum graph parameters using linear optimization
  over the completely positive semidefinite cone.
\newblock {\em SIAM Journal on Optimization}, 25(4):2461--2493, 2015.

\bibitem{Lo78}
L.~Lov{\'a}sz.
\newblock Kneser's conjecture, chromatic number, and homotopy.
\newblock {\em Journal of Combinatorial Theory, Series A}, 25(3):319 -- 324,
  1978.

\bibitem{Lo79}
L.~Lov{\'a}sz.
\newblock On the shannon capacity of a graph.
\newblock {\em IEEE Transactions on Information Theory}, 25(1):1--7, 2006.

\bibitem{MR16}
L.~Man\v{c}inska and D.~E. Roberson.
\newblock Quantum homomorphisms.
\newblock {\em Journal of Combinatorial Theory, Series B}, 118:228--267, 2016.

\bibitem{MSS13}
L.~Man\v{c}inska, G.~Scarpa, and S.~Severini.
\newblock New separations in zero-error channel capacity through projective
  {K}ochen-{S}pecker sets and quantum coloring.
\newblock {\em IEEE Transactions on Information Theory}, 59(6):4025--4032,
  2013.

\bibitem{NFAV15}
M.~Navascu\'es, A.~Feix, M.~Araujo, and T.~V\'ertesi.
\newblock Characterizing finite-dimensional quantum behavior.
\newblock {\em Physical Review A}, 92, 2015.

\bibitem{NPA08}
M.~Navascu{\'e}s, S.~Pironio, and A.~Ac{\'\i}n.
\newblock A convergent hierarchy of semidefinite programs characterizing the
  set of quantum correlations.
\newblock {\em New Journal of Physics}, 10(7):073013, 2008.

\bibitem{NPA12}
M.~Navascu{{\'e}}s, S.~Pironio, and A.~Ac{{\'\i}}n.
\newblock {SDP relaxations for non-commutative polynomial optimization}.
\newblock In M.~F. Anjos and J.~B. Lasserre, editors, {\em Handbook on
  Semidefinite, Conic and Polynomial Optimization}, pages 601--634. Springer,
  2012.

\bibitem{NMVT15}
M.~Navascu\'es and T.~V\'ertesi.
\newblock Bounding the set of finite dimensional quantum correlations.
\newblock {\em Physical Review Letters}, 115(2):020501, 2015.

\bibitem{Nie16}
J.~Nie.
\newblock Symmetric tensor nuclear norms.
\newblock {\em SIAM Journal on Applied Algebra and Geometry}, 1(1):599--625,
  2017.

\bibitem{NC00}
M.~A. Nielsen and I.~L. Chuang.
\newblock {\em Quantum Computation and Quantum Information}.
\newblock Cambridge University Press, 2000.

\bibitem{OP16}
C.~M. Ortiz and V.~I. Paulsen.
\newblock Quantum graph homomorphisms via operator systems.
\newblock {\em Linear Algebra and its Applications}, 497:23--43, 2016.

\bibitem{Oz12}
N.~Ozawa.
\newblock {About the Connes' embedding problem--algebraic approaches}.
\newblock {\em Japanese Journal of Mathematics}, 8(1):147--183, 2013.

\bibitem{PV08}
K~F. P\'al and T.~V\'ertesi.
\newblock {Efficiency of higher-dimensional Hilbert spaces for the violation of
  Bell inequalities}.
\newblock {\em Physical Review A}, 77:042105, 2008.

\bibitem{PV16}
C.~Palazuelos and T.~Vidick.
\newblock Survey on nonlocal games and operator space theory.
\newblock {\em Journal of Mathematical Physics}, 57(1):015220, 2016.

\bibitem{Par00}
P.~A. Parrilo.
\newblock {\em Structured Semidefinite Programs and Semialgebraic Geometry
  Methods in Robustness and Optimization}.
\newblock PhD thesis, Caltech, 2000.

\bibitem{PSSTW16}
V.~I. Paulsen, S.~Severini, D.~Stahlke, I.~G. Todorov, and A.~Winter.
\newblock Estimating quantum chromatic numbers.
\newblock {\em Journal of Functional Analysis}, 270(6):2188--2222, 2016.

\bibitem{NPA10}
S.~Pironio, M.~Navascu{\'e}s, and A.~Ac{\'\i}n.
\newblock Convergent relaxations of polynomial optimization problems with
  noncommuting variables.
\newblock {\em SIAM Journal on Optimization}, 20(5):2157--2180, 2010.

\bibitem{PSVW16}
A.~Prakash, J.~Sikora, A.~Varvitsiotis, and Z.~Wei.
\newblock Completely positive semidefinite rank.
\newblock {\em Mathematical Programming Series A}, 2017.
\newblock To appear.

\bibitem{PV17}
A.~Prakash and A.~Varvitsiotis.
\newblock Matrix factorizations of correlation matrices and applications.
\newblock {\em arXiv:1702.06305}, 2017.

\bibitem{Rob13}
D.~E. Roberson.
\newblock {\em Variations on a Theme: Graph Homomorphisms}.
\newblock PhD thesis, University of Waterloo, 2013.

\bibitem{SV16}
J.~Sikora and A.~Varvitsiotis.
\newblock Linear conic formulations for two-party correlations and values of
  nonlocal games.
\newblock {\em Mathematical Programming}, 162(1):431--463, 2017.

\bibitem{SVW16}
J.~Sikora, A.~Varvitsiotis, and Z.~Wei.
\newblock {Minimum dimension of a Hilbert space needed to generate a quantum
  correlation}.
\newblock {\em Physical Review Letters}, 2016.

\bibitem{Slofstra17}
W.~Slofstra.
\newblock {The set of quantum correlations is not closed}.
\newblock {\em arXiv:1703.08618}, 2017.

\bibitem{Stark15}
C.~Stark.
\newblock {Learning optimal quantum models is NP-hard}.
\newblock {\em arXiv:1510.02800}, 2015.

\bibitem{Szegedy}
M.~Szegedy.
\newblock {A note on the theta number of Lov\'asz and the generalized Delsarte
  bound}.
\newblock In {\em Proceedings of the 35th Annual IEEE Symposium on Foundations
  of Computer Science}, pages 36--39, 1994.

\bibitem{TS15}
G.~Tang and P.~Shah.
\newblock Guaranteed tensor decomposition: A moment approach.
\newblock In {\em Proceedings of the 32nd International Conference on Machine
  Learning}, pages 1491--1500, 2015.

\bibitem{Tsirelson06}
B.~Tsirelson.
\newblock Bell inequalities and operator algebras.
\newblock Technical report, 2006.
\newblock \url{http://www.tau.ac.il/~tsirel/download/bellopalg.pdf}.

\bibitem{Wed}
J.~H.~M. Wedderburn.
\newblock {\em Lectures on Matrices}.
\newblock Dover Publications Inc., 1964.

\bibitem{WCD08}
S.~Wehner, M.~Christandl, and A.~C. Doherty.
\newblock Lower bound on the dimension of a quantum system given measured data.
\newblock {\em Physical Review A}, 78:062112, 2008.

\end{thebibliography}

\appendix 

\section{Synchronous quantum correlations}\label{sec:sync}

We prove the following result by combining proofs from~\cite{SV16} (see also~\cite{MR16}) and~\cite{PSSTW16}.

\propcorrelationsynchronous
\begin{proof}
Suppose first that $(\psi, E_s^a, F_t^b)$ is a realization of $P$ in local dimension $d$. That is,
$\psi$ is a unit vector in $\C^d\otimes\C^d$, $E^a_s,F^b_t$ are $d\times d$ Hermitian positive semidefinite matrices such that 
$\sum_a E^a_s=\sum_b F^b_t=I$ for all $s,t$ and $P(a,b|s,t)=\psi^* (E^a_s\otimes F^b_t)\psi$ for all $(a,b,s,t)\in\Gamma$.
We will show $\cpsdr_\C(A_P) \leq d$. 

The Schmidt decomposition of the unit vector $\psi$  gives nonnegative scalars $\{\lambda_i\}$ and orthonormal bases $\{u_i\}$ and $\{v_i\}$ of $\C^d$ such that $\psi = \smash{\sum_{i=1}^d \sqrt{\lambda_i}} \, u_i \otimes v_i$. If we replace $\psi$ by $\smash{\sum_{i=1}^d \sqrt{\lambda_i}}\, v_i \otimes v_i$ and $E_s^a$ by $U^* E_s^a U$, where $U$ is the unitary matrix for which $u_i = Uv_i$ for all $i$, then we obtain a new realization 
$(\smash{\sum_{i=1}^d \sqrt{\lambda_i}}\, v_i \otimes v_i, U^*E^a_sU,F^b_t)$ of $P$ still in local dimension $d$.
For the simplicity of notation we rename $U^*E^a_sU$ as $E^a_s$.
Then we define the matrices
\[
K = \sum_{i=1}^d \sqrt{\lambda_i} \, v_iv_i^*, \quad X_s^a = K^{1/2} E_s^a K^{1/2}, \quad Y_t^b = K^{1/2} F_t^b K^{1/2}.
\]
By using the identities $\mathrm{vec}(K) = \psi$ and 
\[
\mathrm{vec}(K)^* (E_s^a \otimes F_t^b) \mathrm{vec}(K) = \Tr(K E_s^a K F_t^b) = \Tr(K^{1/2} E_s^a K^{1/2} K^{1/2} F_t^b K^{1/2}),
\]
we see that 
\begin{equation}
\label{sync1}
P(a,b|s,t) = \langle X_s^a, Y_t^b \rangle \quad \text{for all} \quad a,b,s,t,
\end{equation}
and
\begin{equation}
\label{sync2}
\langle K, K \rangle = 1, \quad \sum_a X_s^a = \sum_b Y_t^b = K \quad \text{for all} \quad s, t.
\end{equation}

For each $s$, by applying twice the Cauchy--Schwarz inequality gives
\begin{align*}
1 &= \sum_a P(a,a|s,s) = \sum_a \langle X_s^a, Y_s^a \rangle \leq \sum_a \langle X_s^a, X_s^a\rangle^{1/2} \langle Y_s^a, Y_s^a\rangle^{1/2} \\
&\leq \Big( \sum_a \langle X_s^a, X_s^a \rangle \Big)^{1/2} \Big( \sum_a \langle Y_s^a, Y_s^a \rangle \Big)^{1/2}\\
&\leq \Big\langle \sum_a X_s^a, \sum_a X_s^a \Big\rangle^{1/2} \Big\langle \sum_a Y_s^a, \sum_a Y_s^a \Big\rangle^{1/2} = \langle K, K\rangle = 1.
\end{align*}
Thus all inequalities above are equalities. The first inequality being an equality shows that there exist scalars $\alpha_{s,a}$ such that $X_s^a = \alpha_{s,a} Y_s^a$ for all $a,s$. The second inequality being an equality  shows that there exist scalars $\beta_s$ such that $\|X_s^a\| = \beta_s \|Y_s^a\|$ for all $a,s$. Hence,
\[
\beta_s \|Y_s^a\| = \|X_s^a\| = \|\alpha_{s,a} Y_s^a\| = \alpha_{s,a} \|Y_s^a\| = \alpha_{s,a} \|Y_s^a\| \quad \text{for all} \quad s,a,
\]
which shows $X_s^a = \beta_s Y_s^a$ for all $s$. Since $\sum_a X_s^a = K = \sum_a Y_s^a$, we have $\beta_s = 1$ for all $s$. Thus $X_s^a = Y_s^a$ for all $a,s$.
Therefore, 
\[
(A_P)_{(s,a), (t,b)} = \langle X_s^a, X_t^b \rangle \quad \text{for all} \quad a,b,s,t,
\]
which shows $\cpsdr_\C(A_P) \leq d$.

For the other direction we suppose $\{X_s^a\}$ are 
Hermitian positive semidefinite matrices with the smallest possible size such that $(A_P)_{(s,a),(t,b)} = \langle X_s^a, X_t^b \rangle$ for all $a,s,t,b$. Then,
\[
1 = \sum_{a,b} P(a,b|s,t) = \sum_{a,b} \langle X_s^a, X_t^b \rangle = \Big\langle \sum_a X_s^a, \sum_b X_t^b\Big\rangle \quad \text{for all} \quad s,t,
\]
which shows the existence of a matrix $K$ such that  $K = \sum_a X_s^a$ for all $s$. We have $\langle K, K \rangle = 1$ so that $\mathrm{vec}(K)$ is a unit vector, and since the factorization is smallest possible, $K$ is invertible. Set $E_s^a = K^{-1/2} X_s^a K^{-1/2}$ for all $s,a$, so that $\sum_a E_s^a = I$ for all $s$. Then,
\[
P(a,b|s,t) = (A_P)_{(s,a),(t,b)} = \langle X_s^a, X_t^b \rangle = \mathrm{vec}(K)^* (E_s^a \otimes E_t^b) \mathrm{vec}(K),
\]
which shows $P$ has a realization of local dimension $\cpsdr_\C(A_P)$.
\end{proof}

\end{document}